\documentclass[12pt,leqno]{amsart}
\usepackage{amsmath, amssymb, amscd, amsfonts,    stmaryrd, turnstile, mathrsfs, eucal}

\newtheorem{theorem}{Theorem}[section]
\newtheorem{lemma}[theorem]{Lemma}
\newtheorem{proposition}[theorem]{Proposition}
\newtheorem{corollary}[theorem]{Corollary}

\theoremstyle{definition}
\newtheorem{definition}[theorem]{Definition}
\newtheorem{example}[theorem]{Example}
\newtheorem{remark}[theorem]{Remark}

\begin{document}

\title[Splitting ring extensions ]{Splitting ring extensions}

\author[G. Picavet and M. Picavet]{Gabriel Picavet and Martine Picavet-L'Hermitte}
\address{Math\'ematiques \\
8 Rue du Forez, 63670 - Le Cendre\\
 France}
\email{picavet.mathu (at) orange.fr}

\begin{abstract} The paper deals with ring extensions $R\subseteq S$ and their lattices $[R,S]$ of subextensions and is mainly devoted to FCP extensions (extensions whose lattices are Artinian and Noetherian). The object of the paper is the introduction and the study of elements of the lattices that split in some sense ring extensions. The reason why is that 
this splitting was used in earlier paper without their common nature being recognized. There are some favorable cases allowing to build splitters, mainly when we are dealing with $\mathcal B$-extensions, for example integral extensions. Integral closures and Pr\"ufer hulls of extensions play a dual role. The paper gives many combinatorics results with the explicit computation of the Pr\"ufer hull of an FCP extension. We show that a split extension cannot be pinched, except  trivially. 
\end{abstract} 


\keywords {FCP extension, integral extension, Pr\"ufer extension, almost-Pr\"ufer extension, support of a module, integral closure, Pr\"ufer hull, splitters, $\mathcal B$-extension, pinched}

\maketitle

\section{Introduction and Notation}

 We consider the category of commutative and unital rings, whose    epimorphisms will be involved. If $R\subseteq S$ is a (ring) extension, we denote by $[R,S]$ the set of all $R$-subalgebras of $S$ and set $]R,S[: =[R,S]\setminus\{R,S\}$ (with a similar definition for $[R,S[$ or $]R,S]$). The poset $([R,S], \subseteq )$ is a lattice whose operations are defined by $(U,V) \mapsto U\cap V $ and $(U,V)\mapsto UV$, for $U, V \in[R,S]$, called the lattice of the extension. In the context of the paper, for an extension $R\subseteq S$ and $T\in [R,S]$, the support $\mathrm{Supp}_R(S/T)$ of the $R$-module $S/T$ is a preponderant tool, as well as $\mathrm{MSupp}_R(S/T):= \mathrm{Supp}_R(S/T) \cap \mathrm{Max}(R)$. 
 As a general rule the subscript $R$ is dropped and specified if it is not $R$. 

A {\it chain} of $R$-subalgebras of $S$ is a set of elements of $[R,S]$ that are pairwise comparable with respect to inclusion.   
  We   say that the extension $R\subseteq S$ has FCP (or is an FCP extension) if each chain in $[R,S]$ is finite, or equivalently, the lattice of the extension is Artinian and Noetherian.  In case $[R,S]$ has finitely many elements, we say that the extension has FIP (or is an FIP extension). Dobbs and the authors characterized FCP and FIP extensions \cite{DPP2}. 
 
We start by explaining the ideas that lead us to write this paper. The two books \cite{KZ} and \cite{KK} authored by Knebusch and his co-authors  introduce Pr\"ufer extensions and we use their terminology and results. They are known in the literature as normal pairs but the key point is that an extension $R\subseteq S$ is Pr\"ufer if and only if $R\to T$ is a flat epimorphism  for each $T \in [R,S]$.

We are interested in {\em factorizations} of ring extensions introduced in \cite{KZ} and comprehensively studied in \cite{KK}. An element $U$ of the lattice $[R,S]$ of an extension is called a factor of the extension if $U$ has a complement $V$ (in the sense of lattices); that is, $R =U\cap V$ and $S= UV$. It can be shown that for a Pr\"ufer extension, these two last conditions are equivalent to the following diagram is a pushout and a pullback:
\centerline{$\begin{matrix} 
 {} &      {}       & U &       {}      &{}   \\
 {} &\nearrow & {}  & \searrow & {}  \\
 R &       {}     & {}  &      {}        & S  \\
 {} &\searrow & {}  & \nearrow & {}  \\
 {} &     {}       & V  &      {}        & {}     
 \end{matrix}$} 
 Moreover, since $UV = U+V$ holds for a Pr\"ufer extension, we get that
$0 \to U/R \to  S/R \to V/R \to 0$ is a split exact sequence of $R$-modules. 

Now for an arbitrary extension, the nice properties of complements or factors that hold for Pr\"ufer extensions are no longer true. As our paper focusses on FCP extensions  and complements may appear in a hidden form, a fact we observed in our earlier papers on FCP extensions, it seemed to us that a systematic study was necessary. 

The key idea to get a theory that provides substantial results is the following. We say that a ring extension $R\subseteq S$ is {\em split} at $T\in [R,S]$ if the supports of the $R$-modules $T/R$ and $S/T$ do not meet. The Theorem \ref{4.5} shows that when the extension has FCP and is split at $T$, then $T$ has a unique complement $T^o$,  verifying   $TT^o = T + T^o$. Moreover, there is an order-isomorphism $\varphi:[R,S]\to[R,T]\times[T,S]$ defined by $V \mapsto(V\cap T,VT)$. There is crosswise exchange between the maximal support of $T/R$ and $S/T^o$ and the maximal support of $T^o/R$ and $S/T$. But the reader should note that the proof relies heavily on properties of FCP extensions. The properties of split  extensions  are the subject of Section 4.

It is now time to introduce the definition of a {\em splitter} $\sigma (X)$ at 
a subset $X$ of $\mathrm{MSupp}_R(S/R)$: an element $T$ of $[R,S]$ is a splitter of the extension at $X$ if $X=\mathrm{MSupp}_R(T/R)$ and $X^c=\mathrm{MSupp}_R(S/T)$, where $X^c$ is the complement of $X$ in $\mathrm{MSupp}_R(S/R)$, in which case we set $\sigma (X):= T$, because such a $T$ is unique.
Naturally, each element that splits an extension is a splitter and vice-versa. 
A comprehensive study of splitters is given in Section  4.
With the notation of the above quoted Theorem, we have $T^o =\sigma (X^c)$ if $X = \mathrm{MSupp}(T/R)$. 

We show that splitters exist in case we deal with an FCP $\mathcal B$-extension; that is, an element of $[R,S]$ is (uniquely) determined by an element of $\prod [[R_M,S_M] | M\in \mathrm{MSupp}(S/R)]$. Actually, an FCP extension $R\subseteq S$ is a $\mathcal B$-extension if and only if $R/P$ is a local ring for each $P \in \mathrm{Supp}(S/R)$, this last condition being valid if the extension is integral \cite[ Proposition 2.21]{Pic 10}. (Theorem ~\ref{1.15}) states that an FCP extension is a $\mathcal B$-extension if and only if each subset of its maximal support  has a splitter.
The properties of $\mathcal B$-extensions are studied in Section 3. A chained FCP extension, like a Pr\"ufer-Manis extension, is a $\mathcal B$-extension. This last fact is already an indication that we have to give a strong look at Pr\"ufer extensions, especially because an FCP extension $R\subseteq S$ has an integral closure $\overline R$ which is Pr\"ufer in $S$, {\em i.e.} the extension is quasi-Pr\"ufer, or equivalently the extension is an INC-pair (see \cite{KZ}). Moreover, an FCP extension is a $\mathcal B$-extension if $\overline R \subseteq S$ is a $\mathcal B$-extension. Among other results, the deep and hard Theorem~\ref{6.13}, characterizes Pr\"ufer FCP extensions that are $\mathcal B$-extensions, by the way generalizing Ayache-Jarboui's nice results on normal pairs of domains (note that these authors did not consider $\mathcal B$-extensions).  

Section 5 explores the properties of FCP $\mathcal B$-extensions and their splitters. For an FCP $\mathcal B$-extension $R\subseteq S$ and an element $M$ of its maximal support, we can consider the elementary splitter $\sigma (M)$. We then observe that any splitter of a subset $X$ of its maximal support is the product (in a unique way) of all the elementary splitters $\sigma (M)$ where $M\in X$ (Proposition~\ref{1.155}). It is also interesting to know that $\sigma (X)$ is the intersection of some ``large quotient rings" of the form $R_{[M]}$, where $M$ varies in the complement of $X$. When the extension is integral, this decomposition is related to a primary decomposition of the $R/I$-module $\sigma (X)/R$, where $I$ is the conductor of the extension. Among many other results, we have the following Theorem~\ref{1.161}: If $R\subseteq S$ is an FCP $\mathcal B$-extension, then any element of $[R,S]$ is a splitter if and only if the extension is FIP, locally chained and its lattice is Boolean.

Another principal aim of the paper concerns almost-Pr\"ufer extensions, that are extensions $R \subseteq S$ such that $\widetilde R \subseteq S$ is integral where $\widetilde R$ is the Pr\"ufer hull of the extension \cite{Pic 5}. They are quasi-Pr\"ufer and an FCP extension is almost-Pr\"ufer if and only if it splits at $\widetilde R$ (or at $\overline R$) (see (Proposition~\ref{split})). In Section 6 we characterize FCP extensions that are almost-Pr\"ufer by conditions on length or by order-isomorphisms of lattices associated to the lattices of some subextensions. We also consider in the same way the almost-Pr\"ufer closures of FCP (or FIP) extensions. One of the most striking result is as follows: an FCP extension $R\subseteq S$ is almost-Pr\"ufer if and only if the natural map $[R,\widetilde R]\to [\overline R, S]$ is an order-isomorphism. Moreover, we are able to compute the Pr\"ufer hull of an FCP extension $R\subseteq S$ as an intersection of some large quotient rings and also as the large quotient ring of $R$ with respect to $1+ (R: \overline R)$. All other results are difficult to sum up and the reader is invited to look at Section 6 for more details.

Section 7 contains information on pinched extensions. If an extension is pinched at an element of its lattice, it cannot be split at this element, except it is trivial. The case of an FIP extension is particularly vivid. A special look is given at FCP extensions that are pinched at the integral closure ({\em id est} each element of the lattice is comparable to the integral closure). We develop a new criterion and apply it to Dobbs-Jarboui AV ring pairs  and Ayache recent results.

\section{Some notation and definitions}

A {\em local} ring is here what is called elsewhere a quasi-local ring. As usual, Spec$(R)$ and Max$(R)$ are the set of prime and maximal ideals of a ring $R$. If $R\subseteq S$ is a ring extension and $P\in\mathrm{Spec}(R)$, then $S_P$ is both the localization $S_{R\setminus P}$ as a ring and the localization at $P$ of the $R$-module $S$. 
For an ideal $I$ of a ring $R$, we write $\mathrm{V}(I):=\{P\in\mathrm{Spec}(R)\mid I\subseteq P\}$ and $\mathrm{D}(I)$ for its complement.

The support of an $R$-module $E$ is $\mathrm{Supp}_R(E):=\{P\in\mathrm{Spec}(R)\mid E_P\neq 0\}$, and $\mathrm{MSupp}_R(E):=\mathrm{Supp}_R(E)\cap\mathrm{Max}(R)$. When $R\subseteq S$ is an extension, we will set $\mathrm{Supp}(T/R):=\mathrm{Supp}_R(T/R)$ and $\mathrm{Supp}(S/T):=\mathrm{Supp}_R(S/T)$ for each $T\in [R,S]$, unless otherwise specified. 

We denote by  $(R:S)$  the conductor of the extension $R\subseteq S$ while the integral closure of $R$ in $S$ is denoted by $\overline R^S$ (or by $\overline R$ if no confusion can occur).
  Recall that a ring extension $R\subseteq S$ is called an {\em i-extension} if the natural map $\mathrm{Spec}(S)\to\mathrm{Spec}(R)$ is injective.

\begin{definition}\label{crucial 1} An extension $R\subset S$ is called {\em crucial} if $|\mathrm{MSupp}(S/R)|$ $= 1$. In this case, setting $\mathrm{MSupp}(S/R) = \{M\}$, where $M$ is a maximal ideal of $R$, 
 $M$ is called the {\em crucial (maximal) ideal} $\mathcal{C}(R,S)$ of $R\subset S$ and the extension is called $M$-crucial. 
\end{definition}

We gave a different definition of crucial extensions \cite[Definition 2.10]{Pic 7}; that is, an extension $R\subset S$ is crucial if $|\mathrm{Supp}(S/R)|= 1$. Our new definition is more suitable for this paper. Moreover, $|\mathrm{Supp}(S/R)|= 1$ implies $|\mathrm{MSupp}(S/R)|= 1$.

One of our tools are the minimal (ring) extensions, introduced by Ferrand-Olivier \cite{FO}. An extension $R\subset S$ is called {\em minimal} if $[R, S]=\{R,S\}$, {\em simple} if $S=R[t]$ for some $t\in S$ and locally minimal if $R_M\subset S_M$ is minimal for each $M\in\mathrm{MSupp}(S/R)$. A minimal extension is simple and is either integral (finite) or a flat epimorphism (\cite[Th\'eor\`eme 2.2]{FO}). Three types of minimal integral extensions exist, ramified, decomposed and inert, whose definitions are given by the following Theorem. A minimal extension is crucial.

\begin{theorem}\label{minimal} \cite [Theorem 2.2]{DPP2} Let $R\subset T$ be an extension and $M:=(R: T)$. Then $R\subset T$ is minimal and finite if and only if $M\in\mathrm{Max}(R)$ and one of the following three conditions holds:
\begin{itemize}
\item {\em inert case}: $M\in\mathrm{Max}(T)$ and $R/M\to T/M$ is a minimal field extension.

\item {\em decomposed case}: There exist $M_1,M_2\in\mathrm{Max}(T)$ such that $M= M _1\cap M_2$ and the natural maps $R/M\to T/M_1$ and $R/M\to T/M_2$ are both isomorphisms.

\item {\em ramified case}: There exists $M'\in\mathrm{Max}(T)$ such that ${M'}^2 \subseteq M\subset M',\  [T/M:R/M]=2$, and the natural map $R/M\to T/M'$ is an isomorphism.
\end{itemize}
In each of the above cases, $M=\mathcal{C}(R,T)$ and $R\subset T$ is an $i$-extension when either inert or ramified. 
\end{theorem}

\begin{definition}\label{canonical}(1) Let $R\subset S$ be an integral extension. Then $R\subseteq S$ is called {\em infra-integral} \cite{Pic 1} if all its residual extensions are isomorphisms and {\em subintegral} \cite{S} if $R\subseteq S$ is an infra-integral $i$-extension.  

(2) Let $R\subset S$ be a  ring extension. The {\em seminormalization} ${}_S^+R$ is the largest subextension $T\in[R,S]$ such that $R\subseteq T$ is subintegral, and also the least subextension $T\in[R,S]$ such that $T\subseteq S$ is seminormal. The {\em t-closure}, ${}_S^tR$ is the largest subextension $T\in[R,S]$ such that $R\subseteq T$ is infra-integral, and also the least subextension $T\in[R,S]$ such that $T\subseteq S$ is t-closed. The {\it canonical decomposition} of an arbitrary ring extension $R\subset S$ is $R\subseteq{}_S^+R\subseteq{}_S^tR\subseteq\overline R\subseteq S$. An FCP extension which is either subintegral or t-closed is an $i$-extension (\cite[Proposition 2.10]{Pic 13}).
\end{definition}

We do not give here the definition of a seminormal or t-closed extension (see \cite{S} and \cite{Pic 1} for the definitions), but give a characterization in the context of FCP extensions.

The arguments of the following Propositions have often appeared in earlier papers. We state them now to get a reference.
 We emphasize that the following  result is crucial. 
  
\begin{proposition}\label{1.91}An integral FCP extension $R\subset S$ is seminormal if and only if $(R:S)$ is an intersection of finitely many maximal ideals of $S$ and this intersection is irredundant in $R$ and $S$.
\end{proposition}

\begin{proof} Set $C:=(R:S)$. Since $R\subset S$ is an integral FCP extension, $\mathrm{V}_R(C)$ (resp.$\mathrm{V}_S(C)$) is finite and contained in $\mathrm{Max}(R)$, (resp. $\mathrm{Max}(S)\ (*)$) by \cite[Theorem 4.2]{DPP2}.

Assume first that $R\subset S$ is seminormal and FCP. Then $C$ is a radical ideal of $S$ by \cite[Lemma 4.8]{DPP2}, an intersection of finitely many maximal ideals of $S$ by $(*)$. 

Conversely, assume that $C$ is an intersection of finitely many maximal ideals of $S$. Let $T\in[R,S]\setminus\{R\}$. Then, $C$ is also an ideal of $T$ which is an intersection of finitely many maximal ideals of $T$. But $C\subseteq(R:T)$, so that $(R:T)$ is an intersection of finitely many maximal ideals of $T$ by \cite[Lemma 4.7]{Pic 0}. It follows that $R\subset S$ is seminormal according to \cite[Lemma 4.8]{DPP2}. 

The last result is obvious.
\end{proof}

\begin{proposition}\label{1.9} An FCP integral extension $R\subset S$ is t-closed if and only if $(R:S)$ is an intersection of finitely maximal ideals of $S$ such that $|\mathrm{V}_R((R:S))|=|\mathrm{V}_S((R:S))|$.
 \end{proposition} 
\begin{proof} 
Let $M\in\mathrm{MSupp}(S/R)$. Assume that $R\subset S$ is an FCP t-closed extension. Then, $R\subset S$ is an FCP seminormal extension, so that $(R:S)$ is an intersection of finitely maximal ideals of $S$ by Proposition \ref{1.91}. Moreover, $R\subset S$ is an i-extension by Definition \ref{canonical}, so that $|\mathrm{V}_R((R:S))|=|\mathrm{V}_S((R:S))|$.

Conversely, assume that $(R:S)$ is an intersection of finitely maximal ideals of $S$ such that $|\mathrm{V}_R((R:S))|=|\mathrm{V}_S((R:S))|$. 
 Localizing at any $M\in\mathrm{V}_R((R:S))$, it follows that we can assume that $(R,M)$ is a local ring because $R\subset S$ is t-closed if and only if $R_M\subset S_M$ is t-closed for any $M\in\mathrm{Max}(R)$ by \cite[Th\'eor\`eme 3.15]{Pic 1}. Then $(R:S)=M$ is the only maximal ideal of $S$. Considering the extension $R/M\subset S/M$, \cite[Lemme 3.10]{Pic 1} shows that $R\subset S$ is t-closed, because so is $R/M\subset S/M$.
\end{proof}

The  connection between the above ideas is that if $R\subseteq S$ has FCP, then any maximal (necessarily finite) chain $\mathcal C$ of $R$-subalgebras of $S$, of the form  $R=R_0\subset R_1\subset\cdots\subset R_{n-1}\subset R_n=S$, with {\em length} $\ell(\mathcal C):=n <\infty$, results from juxtaposing $n$ minimal extensions $R_i\subset R_{i+1},\ 0\leq i\leq n-1$. 

For any extension $R\subseteq S$, the {\em length} $\ell[R,S]$ of $[R,S]$ is the supremum of the lengths of chains of $R$-subalgebras of $S$. Notice that if $R\subseteq S$ has FCP, then there {\em does} exist some maximal chain of $R$-subalgebras of $S$ with length $\ell[R,S]$ \cite[Theorem 4.11]{DPP3}. 

We will say that an extension $R\subseteq S$ is {\em chained} if $[R,S]$ is a chain.
 
Recall that an extension $R\subseteq S$ is called {\em Pr\"ufer} if $R\subseteq T$ is a flat epimorphism for each $T\in[R,S]$ (or equivalently, if $R\subseteq S$ is a normal pair \cite[Theorem 5.2, p. 47]{KZ}). A useful result is \cite[Theorem 5.2, p.47-48] {KZ} which states that a ring extension $R \subset S$ is Pr\"ufer if and only if $(R:s)R[s] =R[s]$ for each $s\in S$.
Note that a Pr\"ufer integral extension is an isomorphism. The {\em Pr\"ufer hull} of an extension $R\subseteq S$ is the greatest {\em Pr\"ufer} subextension $\widetilde R$ of $[R,S]$ \cite{Pic 3}. 
  In \cite{Pic 5}, we defined a {\em quasi-Pr\"ufer} extension as an extension that can be factored $R\subseteq R'\subseteq S$, where $R\subseteq R'$ is integral and $R'\subseteq S$ is Pr\"ufer, or equivalently, $\overline R \subseteq S$ is Pr\"ufer. An FCP extension is quasi-Pr\"ufer \cite[Corollary 3.4]{Pic 5}. 
   In \cite{Pic 5}, a minimal flat epimorphism, being Pr\"ufer, is called a {\it Pr\"ufer minimal} extension.
     
An extension $R\subseteq S$ is called {\em almost-Pr\"ufer} if $\widetilde R\subseteq S$ is integral, or equivalently, when $R\subseteq S$ is FCP, whence quasi-Pr\"ufer, if $S=\widetilde R\overline R$ \cite[Theorem 4.6]{Pic 5}. An almost-Pr\"ufer extension  is quasi-Pr\"ufer.  
 
Finally, $|X|$ is the cardinality of a set $X$, $\subset$ denotes proper inclusion and if $X\subseteq Y$, we denote by $X^c$ its complement in $Y$, and for a positive integer $n$, we set $\mathbb{N}_n:=\{1,\ldots,n\}$. 

  \section{$\mathcal B$-extensions}

Let $R\subset S$ be an FCP extension with $\mathrm{MSupp}(S/R)=\{M_i\}_{i=1}^n$. 
 Recall that $|\mathrm{MSupp}(S/R)|<\infty$ by \cite[Corollary 3.2]{DPP2}.
According to \cite[Theorem 3.6]{DPP2}, the map $\varphi:[R,S]\to\prod_{i=1}^n[R_{M_i},S_{M_i}]$ defined by $\varphi(T):=(T_{M_i})_{i=1}^n$ is injective. In \cite[the paragraph before Proposition 2.21]{Pic 10}, we say that $R\subset S$ is a $\mathcal B$-extension if $\varphi$ is bijective ($\mathcal B$ stands for bijective). In this case, $\varphi$ is an order isomorphism (for the product order on $\prod_{i=1}^n[R_{M_i},S_{M_i}]$).  
  We characterized an FCP $\mathcal B$-extension as follows, adding  now a property of the length of  such an extension:
  
 \begin{proposition}\label{6.5} \cite[ Proposition 2.21]{Pic 10} An FCP extension $R\subseteq S$ is a $\mathcal B$-extension if and only if $R/P$ is  local for each $P\in\mathrm{Supp}(S/R)$ and in this case
 $\ell[R,S]=\sum(\ell[R_M,S_M]\mid M\in \mathrm{MSupp}(S/R))$. 

\noindent The above ``local" condition on the factor domains $R/P$ holds in case $\mathrm{Supp}(S/R)\subseteq\mathrm{Max}(R)$, and, in particular, if $R\subset S$ is  integral. 

\end{proposition}
\begin{proof} We have to prove the equation. Set $\mathrm{MSupp}(S/R):= $ $\{M_1,\ldots,M_n\}$. If $\varphi$ is bijective, $\ell[R,S]=\ell(\prod_{i=1}^n[R_{M_i},S_{M_i}])=\sum_{i=1}^n\ell[R_{M_i},S_{M_i}]$ by \cite[Proposition 3.7(d)]{DPP2}, (there is a misprint in the quoted reference).  
\end{proof}

Let $T\in[R,S]$. Then, $T'\in[R,S]$ is called a {\em complement} of $T$ if $T\cap T'=R$ and $TT'=S$. If $T'$ is the unique complement of $T$ in $[R,S]$, we write $T'=T^o$. 

An extension $R\subseteq S$ is called {\em Boolean} if $([R,S],\cap,\cdot)$ is a  distributive lattice (that is such that intersection and product are each distributive with respect to the other) and such that each $T\in[R,S]$ has a (necessarily unique) complement.

\begin{corollary}\label{6.51} An FCP  extension $R\subseteq S$ is locally minimal if and only if $R\subset S$ is an FIP $\mathcal B$-extension such that $|[R,S]|=2^{|\mathrm{MSupp}(S/R)|}$. If these conditions hold, then $R\subset S$ is a Boolean extension such that  $\ell[R,S]=|\mathrm{MSupp}(S/R)|$.
\end{corollary}
\begin{proof} Assume first that $R\subseteq S$ is locally minimal. We show that $R\subset S$ is a $\mathcal B$-extension. Let $P\in\mathrm{Supp}(S/R)$ and assume that $P\not\in\mathrm{Max}(R)$. Then, there exists $M\in\mathrm{MSupp}(S/R)$ such that $P\subset M$. In particular, there exist $T,U\in[R,S]$ such that $T\subset U$ is minimal with $P:=\mathcal{C}(T,U)\cap R$. But, $R_M\subseteq T_M\subset U_M\subseteq S_M$ is minimal with $T_M\neq U_M$ because $T_P\neq U_P$. It follows that $R_M=T_M$ and $U_M=S_M$. Then, $M\in\mathrm{MSupp}(U/R)$ and there exists a unique $N\in\mathrm{Max}(T)$ lying over $M$, so that $T_M=T_N$ and $U_M=U_N$. This implies that $N\in\mathrm{MSupp}_T(U/T)$, so that $N=\mathcal{C}(T,U)$, a contradiction with $P=\mathcal{C}(T,U)\cap R\subset M=N\cap R$. Then, $R\subset S$ is a $\mathcal B$-extension. It follows that $R\subset S$ is an FIP extension since $|[R_M,S_M]|=2$ for any $M\in\mathrm{MSupp}(S/R)$ and the definition of a $\mathcal B$-extension gives $|[R,S]|=2^{|\mathrm{MSupp}(S/R)|}$.

Conversely, assume that $R\subset S$ is an FIP $\mathcal B$-extension such that $|[R,S]|=2^{|\mathrm{MSupp}(S/R)|}$. It follows that $|[R_M,S_M]|=2$ for each $M\in\mathrm{MSupp}(S/R)$, so that $R\subseteq S$ is locally minimal.

If these conditions hold, then $R\subseteq S$ is Boolean according to \cite[Proposition 3.5 and Example 3.3 (1)]{Pic 10}. By Proposition \ref{6.5}, we get that $\ell[R,S]=|\mathrm{MSupp}(S/R)|$.
\end{proof}

The statement of Corollary \ref{6.51} admits  equivalent statements if the extension is either Pr\"ufer  (Proposition \ref{6.12}) or split  (Theorem~\ref{1.161}).

The local condition for $\mathcal B$-extension holds  for every prime ideals of  pm-rings,  introduced  by De Marco and  Orsatti in \cite{MO} (they are also called Gelfand rings by some other authors).  A ring $R$ is a {\em pm-ring} if every prime ideal of $R$ is contained in a unique maximal ideal of $R$. 
 Of course, a local ring is a pm-ring.

We deduce the obvious following result: 

\begin{proposition}\label{6.6} If $R$ is a pm-ring, then any FCP extension $R\subseteq S$ is a $\mathcal B$-extension. 
\end{proposition}

 We exhibit  in Example \ref{0.23}  an FCP $\mathcal B$-extension $R\subseteq S$ such that $R$ is not a pm-ring.

 \begin{proposition}\label{6.71} A chained  FCP extension  $R\subseteq S$  is a $\mathcal B$-extension. 
 For example, a  Pr\"ufer-Manis FCP extension is a $\mathcal B$-extension.
 \end{proposition}

\begin{proof} Proposition \ref{6.5} gives the answer: let $P\in \mathrm{Supp}(S/R)$ and assume there exist $M,M'\in\mathrm{Max}(R),\ M\neq M'$, such that $P\subseteq M\cap M'$. According to \cite[Lemma 1.8]{Pic 6}, there exist $T,T'\in[R,S],\ T'\neq T$ such that $R\subset T$ and $R\subset T'$ are minimal with $M:=\mathcal{C}(R,T)$ and $M':=\mathcal{C}(R,T')$, a contradiction with $R\subseteq S$ is chained. Then, $P$ is contained in a unique maximal ideal of $R$ and $R\subseteq S$ is a $\mathcal B$-extension.

 Assume that $R\subseteq S$ is a Pr\"ufer-Manis FCP extension, then $R\subseteq S$ is chained by \cite[Theorem 3.1, page 187]{KZ}. Hence, $R\subseteq S$ is a $\mathcal B$-extension since FCP.
\end{proof}

By using the next result, we can reduce in some contexts the characterization of FCP $\mathcal B$-extensions to the case of Pr\"ufer extensions. 

\begin{proposition}\label{6.8} Let $R\subseteq S$ be an FCP extension.
\begin{enumerate}
\item If $\overline R\subseteq S$ is a $\mathcal B$-extension, so is $R\subseteq S$.
\item If $R\subseteq S$ is a $\mathcal B$-extension and $R\subseteq \overline R$ is   an i-extension, then $\overline R\subseteq S$ is a $\mathcal B$-extension.
\item If $R\subseteq S$ is a $\mathcal B$-extension and $T\in[R,S]$, then $R\subseteq T$ is a $\mathcal B$-extension.
\end{enumerate} 
\end{proposition}

\begin{proof} We use the characterization of Proposition \ref{6.5}.

(1) Assume that $\overline R\subseteq S$ is a $\mathcal B$-extension and let $P\in\mathrm{Supp}(S/R)$. To show that $R\subseteq S$ is a $\mathcal B$-extension, it is enough to consider $P\not\in\mathrm{Max}(R)$, so that $P\not\in\mathrm{Supp}(\overline R/R)$, since $\mathrm{Supp}(\overline R/R)\subseteq\mathrm{Max}(R)$. This implies that $P\in\mathrm{Supp}_R(S/\overline R)$. But $P\not\in\mathrm{Supp}(\overline R/R)$ implies that there is a unique $Q\in \mathrm{Spec}(\overline R)$ lying over $P$, so that ${\overline R}_P={\overline R}_Q$ by \cite[Lemma 2.4]{DPP2}. In particular, $Q\in \mathrm{Supp}_{\overline R}(S/\overline R)$. Let $\{R_i\}_{i=0}^n$ be a maximal chain of $[\overline R,S]$. There exists some $i\in\mathbb{N}_n$ such that $Q=\overline R\cap \mathcal{C}(R_{i-1},R_i)$. Since $\overline R\subseteq S$ is a $\mathcal B$-extension, $Q$ is contained in a unique maximal ideal of $\overline R$. Then, $P=Q\cap R$  is contained in a unique maximal ideal of $R$ and $ R\subseteq S$ is a $\mathcal B$-extension.

(2) Assume that $R\subseteq S$ is a $\mathcal B$-extension and that $R\subseteq \overline R$ is an $i$-extension. Let $Q\in\mathrm{Supp}_{\overline R}(S/\overline R)$ and set $P:=Q\cap R$, so that $P\in\mathrm{Supp}_R(S/\overline R)\subseteq \mathrm{Supp}(S/R)$. The assumption gives that $P$ is contained in a unique maximal ideal $M$ of $R$. But $R\subseteq \overline R$ being an $i$-extension, there is a unique $N\in\mathrm{Max}(\overline R)$ lying above $M$. Then $N$ is the unique maximal ideal of $\overline R$ containing $Q$ and $\overline R\subseteq S$ is a $\mathcal B$-extension. 

(3) Obvious since $\mathrm{Supp}(T/R)\subseteq\mathrm{Supp}(S/R)$.
\end{proof}

\begin{remark}\label{6.9} (1) 
Let $R\subseteq S$ be an FCP extension. If $\overline R$ is local, {\em i.e.} $R\subseteq S$ is unbranched, then $\overline R\subseteq S$ is chained, according to \cite[Theorem 6.10]{DPP2}, and then a $\mathcal B$-extension by Proposition \ref{6.71}.
 An extension $R\subset S$ is said {\it almost unbranched} if each $T\in[R,\overline R[$ is a local ring.

(2) The $i$-extension hypothesis of Proposition \ref{6.8} (2) cannot be suppressed. We now exhibit an example of an FCP $\mathcal B$-extension $R\subseteq S$ such that $\overline R\subseteq S$ is not a $\mathcal B$-extension. Let $R\subseteq S$ be an extension such that $(R,M)$ is a two-dimensional local ring, with $R\subseteq\overline R$ minimal decomposed, so that $R\subseteq\overline R$ is not an $i$-extension. Set $\mathrm{Max}(\overline R)=:\{M_1,M_2\}$. Assume, moreover, that $\overline R\subseteq S$ is a Pr\"ufer extension of length 2 and that there exists $P\in\mathrm{Supp}(S/R),\ P\neq M$. Since $P\not\in\mathrm{Supp}(\overline R/R)$, this implies that there is a unique $Q\in\mathrm{Spec}(\overline R)$ lying over $P$. In particular, we have $P\in\mathrm{Supp}(S/\overline R)$, which leads to $Q\in\mathrm{Supp}_{\overline R}(S/\overline R)$. If $Q\subseteq M_1\cap M_2$, then $\overline R\subseteq S$ is not a $\mathcal B$-extension. It is enough to consider a two-dimensional Pr\"ufer domain $R$ with exactly two height-$2$ maximal ideals, $N_1$ and $N_2$, both of then containing the unique height-$1$ prime ideal $P$ of $R$. Such situation exists (see  \cite[Theorem 3.1]{Lew}).

This also happens when $R\subseteq\overline R$ satisfies going down. In this case, since $P\subset M$, with $M_1,M_2$ lying over $M$, there exist $Q_1,Q_2\in\mathrm{Spec}(\overline R)$ lying over $P$ and such that $Q_i\subset M_i$ for $i=1,2$. But $R\subseteq\overline R$ being minimal, we get $R_P=\overline R_P$ so that there is a unique $Q\in \mathrm{Spec}(\overline R)$ lying over $P$, giving $Q=Q_1=Q_2$ which leads to $Q\subseteq M_1\cap M_2$.

We   meet this situation in the following example.
\end{remark}

\begin{example}\label{6.10} Consider the ring of integer valued polynomials $T':=\mathrm{Int}(\mathbb Z)$. It is known that $T'$ is a two-dimensional Pr\"ufer domain whose prime ideals are of two types \cite[Proposition V.2.7]{CC} and \cite[Example, page 293]{Lu}. Set $q(X):=X^2-X+p\in{\mathbb Q}[X]$, where $p$ is an odd prime number. Then, \cite[Exercise 12, page 116]{CC} says that $P_q:=q{\mathbb Q}[X]\cap T'$ is a height one prime ideal of $T'$ lying above $0$ in $\mathbb Z$, and contained in two distinct maximal ideals $M_1'$ and $M_2'$ of $T'$, lying both above $p\mathbb Z$. Set $\Sigma:=T'\setminus(M_1'\cup M_2')$, which is a multiplicatively closed subset, and $T:=T_{\Sigma},\  Q:=(P_q)_{\Sigma }$ and $ M_i:=(M_i')_{\Sigma}$ for $i=1,2$. Then, $T$ is a two-dimensional Pr\"ufer domain with two maximal ideals $M_1$ and $M_2$. Moreover, $T=T_{M_1}\cap T_{M_2}$, where $V_i:=T_{M_i}$ is a two-dimensional valuation domain with maximal ideal $N_i:=M_iV_i$, for $i=1,2$ and $Q$ is a prime ideal of $T$ contained in $M_1$ and $M_2$. We have $N_i\cap \mathbb Z=M_i\cap \mathbb Z=p \mathbb Z$ and $V_i/N_i\cong T/M_i\cong \mathbb Z/p\mathbb Z$ for $i=1,2$. Set $M:=M_1\cap M_2$ and $R:=\mathbb Z+M$, which gives $M\cap\mathbb Z=p\mathbb Z$. It follows that $R/M\cong T/M_i\cong\mathbb Z/p\mathbb Z$ for $i=1,2$, so that $R\subset T$ is a minimal decomposed extension, with $(R,M)$ a local domain. Moreover, since $T$ is a Pr\"ufer domain, $Q$ is the unique non maximal prime ideal of height one contained in $M_1$ and $M_2$, giving that $T_Q$ is a one-dimensional valuation domain. Setting $P:=Q\cap R$, this yields that $P$ is the unique non maximal prime ideal of height one contained in $M$. As $\mathrm{Spec}(T)=\{0,Q,M_1,M_2\}$, we get the following commutative diagram,  where 
$K:={\mathbb Q}(X)$:

\centerline{$\begin{matrix}
 {} & {}  &  {} &      {}      & V_1 &       {}      &{}       & {}  & {} \\
 {} & {}  & {} &\nearrow & {}     & \searrow & {}      & {}  & {} \\
R & \to & T &       {}      & {}     &      {}        & T_Q & \to & K \\
 {} & {}  & {} &\searrow & {}     & \nearrow & {}     & {} & {} \\
 {} &{}  &  {} &     {}       & V_2 &      {}        & {}     & {} & {}
 \end{matrix}$}
We obtain that $\mathrm{Supp}_T(K/T)=\{Q,M_1,M_2\}$ and $\mathrm{Supp}_ R(K/ R)=\{P,M\}$. At last, $T$ is the integral closure of $R$ in $K$, giving that $R\subseteq K$ is a $\mathcal B$-extension and $T\subseteq K$ is not a $\mathcal B$-extension since the prime ideal $Q$ of $T$ is in $\mathrm{Supp}_T(K/T)$ and is contained in two maximal ideals of $T$.
\end{example}

As we said in Proposition \ref{6.5}, an FCP integral extension is a $\mathcal B$-extension. Let $R\subseteq S$ be an FCP extension. If $\overline R\subseteq S$ is a $\mathcal B$-extension, by Proposition \ref{6.8}, we deduce that $R\subseteq S$ is a $\mathcal B$-extension. In the following, we characterize FCP Pr\"ufer $\mathcal B$-extensions. Recall \cite[Theorem 6.3]{DPP2} that an FCP Pr\"ufer extension has FIP. At the same time, we take the opportunity to get some additional results about the cardinality of the set of intermediate rings in FCP Pr\"ufer extensions, after those obtained in \cite{DPP2} and \cite{DPP3}. We give also results generalizing some results that Ayache and Jarboui get for extensions of integral domains in \cite{AJar}.  

The following proposition allows us to generalize \cite[Theorem 6.10]{DPP2} and makes \cite[Corollary 2.7]{AJar} valid for an arbitrary FCP Pr\"ufer  extension. Although the context of Ayache-Jarboui's result seems different from ours, they are in fact the same but our proof is different.
   
According to \cite[Proposition 6.9]{DPP2}, an integrally closed extension $R\subseteq S$ has FCP if and only if it is a normal pair such that $\mathrm{Supp}(S/R)$ is finite if and only if it is an FCP Pr\"ufer  extension. 
  
We recall the following definition and notation. A poset $(X,\leq)$ is called a {\em tree} if $x_1,x_2\leq x_3$ in $X$ implies that $x_1$ and $x_2$ are comparable (with respect to $\leq$). 
  
A subset $Y$ of $X$ is called an {\em antichain} if no two distinct elements of $Y$ are comparable. Let $R\subset S$ be an FCP Pr\"ufer extension. Set $m:=|\mathrm{MSupp}(S/R)|$. For each $k\in\mathbb N_m$, let $\Gamma _k$ be the set of all antichains of $\mathrm{Supp}(S/R)$ that have cardinality $k$. 

\begin{proposition}\label{6.11} Let $R\subset S$ be an FCP Pr\"ufer  extension, (whence FIP). The following statements are equivalent.
\begin{enumerate}
\item $R\subseteq S$ is chained;
\item There exists a unique $T\in[R,S]$ such that $R\subset T$ is  minimal; 
\item $\mathrm {Supp}(S/R)$ is linearly ordered;
 \item $|[R,S]|=1+|\mathrm {Supp}(S/R)|$.
\end{enumerate}

If one of the above conditions is verified, then $R\subset S$ is a $\mathcal B$-extension.
\end{proposition} 

\begin{proof} (1) $\Rightarrow$ (2) Obvious.

(2) $\Rightarrow$ (3) Set $M:=\mathcal{C}(R,T)\in\mathrm{Max}(R)$, where $T$ is the unique element of $[R,S]$ such that $R\subset T$ is minimal. We claim that $\mathrm{MSupp}(S/R)$

\noindent $=\{M\}$. Assume there exists $M'\in\mathrm{MSupp}(S/R)\setminus\{M\}$. According to \cite[Lemma 1.8]{Pic 6}, there exists $T'\in[R,S],\ T'\neq T$ such that $R\subset T'$ is minimal with $M':=\mathcal{C}(R,T')$, a contradiction. It follows that $\mathrm{MSupp}(S/R)=\{M\}$ and $P\subseteq M$ for any $P\in\mathrm{Supp}(S/R)$. 
  By \cite[Proposition 4.2(b)]{DPP3}, $\mathrm{Supp}(S/R)$ is a tree, and then is linearly ordered since it has only one maximal element. 

(3) $\Rightarrow$ (4) We keep notation of \cite[Theorem 4.3 and Proposition 4.2]{DPP3}, with $m:=|\mathrm{MSupp}(S/R)|$. Then, we have $|\mathrm{MSupp}(S/R)|=1$, so that $m=1$ and $|\Gamma_1|=|\mathrm{Supp}(S/R)|$ because $\Gamma_1=\mathrm{Supp}(S/R)$ since any antichain of $\mathrm{Supp}(S/R)$ has only one element. By \cite[Theorem 4.3]{DPP3}, we get $|[R,S]|=1+|\mathrm {Supp}(S/R)|$.

(4) $\Rightarrow$ (1) because $|\mathrm{Supp}(S/R)|=\ell[R,S]$ by \cite[Proposition 6.12]{DPP2}, so that $|[R,S]|=1+\ell[R,S]$, and there is a unique maximal chain in $[R,S]$. 

The last statement follows from Proposition~\ref{6.71}.
\end{proof}

As a second proposition, we now give a generalization to arbitrary rings  of Ayache-Jarboui's result \cite[Theorem 3.1]{AJar}. Our proof is  completely different.

\begin{proposition}\label{6.12} Let $R\subset S$ be an FCP Pr\"ufer  extension. Then, the following conditions are equivalent:
\begin{enumerate}
\item $R\subset S$ is locally minimal;

\item $\mathrm{Supp}(S/R)\subseteq\mathrm{Max}(R)$;

\item $|[R,S]|=2^{|\mathrm {Supp}(S/R)|}$;

\item  
  $R\subset S$ is a Boolean FIP $\mathcal B$-extension.
  \end{enumerate}
\end{proposition}

\begin{proof} (1) $\Leftrightarrow$ (2) Since $R\subset S$ is an FCP Pr\"ufer extension, so is $R_M\subset S_M$ for any $M\in\mathrm{Supp}(S/R)$. It follows from \cite[Proposition 6.12]{DPP2} that $\ell[R_M,S_M]=|\mathrm{Supp}(S_M/R_M)|$ for any $M\in\mathrm{Supp}(S/R)$. Then, $R\subset S$ is locally minimal $\Leftrightarrow$ for any $M\in\mathrm{Supp}(S/R),\ \ell[R_M,S_M]=1=|\mathrm{Supp}(S_M/R_M)|\Leftrightarrow$ for any $M\in\mathrm{Supp}(S/R),\ \mathrm{Supp}(S_M/R_M)=\{MR_M\}\Leftrightarrow \mathrm{Supp}(S/R)\subseteq\mathrm{Max}(R)$.

(1) $\Rightarrow$ (3) by Corollary ~\ref{6.51} because  $\mathrm{Supp}(S/R)\subseteq\mathrm{Max}(R)$.

(3) $\Rightarrow$ (2) Assume that $|[R,S]|=2^{|\mathrm{Supp}(S/R)|}$. We keep the  notation of \cite[Theorem 4.3]{DPP3}. Set $r:=|\mathrm{Supp}(S/R)|$ and $n:=|\mathrm{MSupp}(S/R)|$, so that $r\geq n$. It follows that $|[R,S]|=2^r=(1+1)^r=\sum_{k=0}^rC_r^k=1+\sum_{k=1}^rC_r^k$. But $C_r^k$ is the number of subsets of $\mathrm{Supp}(S/R)$ whose cardinality is $k$. Moreover, for each $k\in\{0,\ldots,n\}$, we have $|\Gamma_k|\leq C_r^k$. Since \cite[Theorem 4.3]{DPP3} gives that $|[R,S]|=1+\sum_{i=1}^n|\Gamma_i|$, we get that $\sum_{k=1}^rC_r^k=\sum_{k=1}^n|\Gamma_k|\leq\sum_{k=1}^nC_r^k$. It follows that $r=n$, giving $\mathrm{MSupp}(S/R)=\mathrm{Supp}(S/R)$ and any element of $\mathrm {Supp}(S/R)$ is a maximal ideal of $R$. Hence, $\mathrm {Supp}(S/R)\subseteq\mathrm{Max}(R)$. 

(1) $\Rightarrow$ (4) $R\subset S$ is a Boolean FIP $\mathcal B$-extension by Corollary ~\ref{6.51}.

(4) $\Rightarrow$ (2) by \cite[Proposition 3.21]{Pic 10} because $R\subset S$ is a Boolean FIP Pr\"ufer extension.
\end{proof}

The next result characterizes FCP Pr\"ufer $\mathcal B$-extensions and generalizes the Ayache-Jarboui's  result \cite[Theorem 3.3]{AJar}. 

\begin{theorem}\label{6.13} Let $R\subset S$ be an FCP Pr\"ufer  extension (whence FIP), $\mathrm{MSupp}(S/R):=$ $\{M_1,\ldots,M_n\}$ and  $\{P_1,\ldots,P_r\}$  the set of minimal elements of $\mathrm{Supp}(S/R)$. Then the following statements are equivalent:

\begin{enumerate}
\item $R\subset S$ is a $\mathcal  B$-extension. 

\item $|[R,S]|={\prod}_{i=1}^n|[R_{M_i},S_{M_i}]|$.

\item $\mathrm{V}(P)$ is linearly ordered for each minimal element $P$ of $\mathrm {Supp}(S/R)$.

\item $|[R,S]|={\prod}_{i=1}^n(2+\mathrm{ht}(M_i/P_{j_i}))$, where $P_{j_i}$ is the only minimal element of $\mathrm{Supp}(S/R)$ contained in $M_i$.

\item $|\mathrm{Supp}(S/R)|={\sum}_{i=1}^n(1+\mathrm{ht}(M_i/P_{j_i}))$, where $P_{j_i}$ is the only minimal element of $\mathrm{Supp}(S/R)$ contained in $M_i$.
\end{enumerate}
\end{theorem}

\begin{proof} First, by \cite[Proposition 4.2(b)]{DPP3} which says that $\mathrm{Supp}(S/R)$ is a tree, each $M_i$ contains only one $P_j$, for $i\in\mathbb N_n$. Let $\Psi:\mathrm{MSupp}(S/R)\to\{P_1,\ldots,P_r\}$ be the map defined by $\Psi(M_i):=P_{j_i}$, where $P_{j_i}$ is the only minimal element of $\mathrm{Supp}(S/R)$ contained in $M_i$.
The map $\varphi:[R,S]\to{\prod}_{i=1}^n[R_{M_i},S_{M_i}]$ defined by $\varphi(T)=(T_{M_1},\ldots,T_{M_n})$, is injective by \cite[Theorem 3.6]{DPP2}.

(1) $\Leftrightarrow$ (2) since $\varphi$ is injective.

(1) $\Rightarrow$ (3) Assume that there is a minimal element $P$ of $\mathrm{Supp}(S/R)$ such that $\mathrm{V}(P)$ is not linearly ordered. There exist $P'_1,P'_2\in\mathrm{V}(P)$ which are incomparable. It follows that they are not contained in a same maximal ideal of $R$ because $\mathrm {Supp}(S/R)$ is a tree. So, we may assume that there exist $M_1,M_2\in\mathrm{V}(P)\cap\mathrm{MSupp}(S/R)$, with $M_1\neq M_2$, a contradiction by Proposition \ref{6.5} since $\varphi$ is bijective. Then, $\mathrm{V}(P)$ is linearly ordered, and this holds  for each minimal element $P$ of $\mathrm {Supp}(S/R)$.

(3) $\Rightarrow$ (1) Assume that $\mathrm{V}(P)$ is linearly ordered for each minimal element $P$ of $\mathrm{Supp}(S/R)$. Then, $R/Q$ is local for each $Q\in\mathrm{Supp}(S/R)$ and Proposition \ref{6.5} shows that  $\varphi$ is bijective.

(2) $\Leftrightarrow$ (4) Let $i\in\mathbb N_n$. Since $R_{M_i}$ is local, we have $R_{M_i}\subseteq S_{M_i}$ chained according to  \cite[Theorem 6.10]{DPP2}. Moreover, $|[R_{M_ i},S_{M_i}]|=1+|\mathrm {Supp}_{R_{M_i}}(S_{M_i}/R_{M_i})|$ and $\mathrm{Supp}_{R_{M_i}}(S_{M_i}/R_{M_i})$ is linearly ordered by Proposition~\ref{6.11}. Its smallest element is of the form $P_{j_i}R_{M_i}$, where $P_{j_i}=\Psi(M_i)$. Then, $|\mathrm{Supp}_{R_{M_i}}(S_{M_ i}/R_{M_i})|=1+\mathrm{ht}(M_i/P_{j_i})$. It follows that $|[R_{M_i},S_{M_i}]|=2+\mathrm{ht}(M_i/P_{j_i})$ and (4) $\Leftrightarrow|[R,S]|={\prod}_{i=1}^n(2+\mathrm{ht}(M_i/P_{j_i}))={\prod}_{i=1}^n|[R_{M_i},S_{M_i}]|\Leftrightarrow$ (2). 

(3) $\Leftrightarrow$ (5) By \cite[Proposition 4.2(c)]{DPP3}, $\cup_{j=1}^r\mathrm{V}_R(P_j)$ is a partition of $\mathrm{Supp}(S/R)$. It follows that $|\mathrm{Supp}(S/R)|={\sum}_{j=1}^r|\mathrm{V}_R(P_ j)|$. Since $\Psi$ is surjective, we get that $r\leq n$. For each $j\in\mathbb N_r$, set $\{M_{j,1},\ldots,M_{j,s_j}\}:=\mathrm{V}_R(P_ j)\cap\mathrm{MSupp}(S/R)=\Psi^{-1}(\{P_j\})$ and ${\sum}_{j=1}^r(\sum_{k=1}^{s_j}[1+\mathrm{ht}(M_{j,k}/P_j)])$

\noindent $=:A$. 
As $\mathrm{Supp}(S/R)$ is a tree, there do not exist $j,j'\in\mathbb N_r,\ j\neq j'$ such that $M_{j,k}=M_{j',k'}$ for any $(k,k')\in\mathbb N_{s_j}\times\mathbb N_{s_{j'}}$.

Then, $A=\sum_{i=1}^n[1+\mathrm{ht}(M_i/P_{j_i})]$, where $P_{j_i}=\Psi(M_i)$ and because each $M_i$ appears one and only one time in $A$. For $j\in\mathbb N_r$, we have $|\mathrm{V}_R(P_j)|\leq\sum_{k=1}^{s_j}[1+\mathrm{ht}(M_{j,k}/P_ j)]$, with equality if and only if $s_j=1$, since $P_j$ is counted in $\mathrm{ht}(M_{j,k}/P_j)$ for each $M_ {j,k}\in\Psi^{-1}(\{P_j\})$. It follows that $|\mathrm{V}_R(P_j)|=\sum_{k=1}^{s_j}[1+\mathrm{ht}(M_{j,k}/P_ j)]$ if and only if $P_j$ is contained in a unique maximal ideal $M_ {j,k}$ if and only if $\mathrm{V}_R(P_j)$ is linearly ordered because $\mathrm {Supp}(S/R)$ is a tree. Now, $|\mathrm{Supp}(S/R)|={\sum}_{j=1}^r|\mathrm{V}_R(P_j)|\leq{\sum}_{j=1}^r(\sum_{k=1}^{s_j}[1+\mathrm {ht}(M_{j,k}/P_j)])=A$. But, $A=\sum_{i=1}^n[1+\mathrm{ht}(M_i/P_{j_i})]$.

To conclude, (5) $\Leftrightarrow|\mathrm{Supp}(S/R)|=A\Leftrightarrow|\mathrm{V}_R(P_j)|=\sum_{k=1}^{s_j}[1+\mathrm{ht}(M_{j,k}/P_j)]$ for each minimal element $P_j$ of $\mathrm{Supp}(S/R)\Leftrightarrow\mathrm{V}_R(P_j)$ is linearly ordered for each minimal element $P_j$ of $\mathrm {Supp}(S/R)\Leftrightarrow$  (3).
\end{proof}   

\section{General properties of splitters}
Given a ring extension $R\subseteq S$ and $T\in[R,S]$, it is easy to see that $\mathrm{MSupp}_R(S/R)=\mathrm{MSupp}_R(T/R)\cup \mathrm{MSupp}_R(S/T)$. 
The case when this equation defines a partition of $\mathrm{Supp}_R(S/R)$ has a great interest, because we will show that in some contexts, we can reduce our results to this situation.

In the following, we denote by $X^c$ the complement  of a subset $X$ of $\mathrm{MSupp}_R(S/R)$ in $\mathrm{MSupp}_R(S/R)$.  

\begin{lemma}\label{split1}Let $R\subset S$ be a ring extension. Assume that there is some $T\in[R,S]$ such that  $\mathrm{MSupp}(S/T)\cap\mathrm{MSupp}(T/R)=\emptyset$. Then, $T$ is the unique  $T'\in[R,S]$ such that $\mathrm{MSupp}(S/T')=\mathrm{MSupp}(S/T)$ and 

\noindent$\mathrm{MSupp}(T/R)=\mathrm{MSupp}(T'/R)$.
\end{lemma}

\begin{proof}Set $X:=\mathrm{MSupp}(T/R)$ and $Y:=X^c$, so that $Y=\mathrm{MSupp}(S/T)$. Let $T'\in[R,S]$ be such that $\mathrm{MSupp}(T'/R)=X$ and $\mathrm{MSupp}(S/T')=Y$. Since $X\cap Y=\emptyset$ and $\mathrm{MSupp}(S/R)=X\cup Y$, we get $T_M=T'_M=S_M$ for any $M\in X$ and $T_M=T'_M=R_M$ for any $M\in Y$. Moreover, since $R_M=T_M=T'_M=S_M$ for any $M\in\mathrm{Max}(R)\setminus( X\cup Y)$, it follows that $T'=T$. 
\end{proof}

\begin{definition}\label{split0}If $R\subset S$ is an extension and $T\in]R,S[$, we say that the extension {\em splits} at $T$, if $\mathrm{MSupp}(S/T)\cap\mathrm{MSupp}(T/R)=\emptyset$. 
\end{definition}
 
 In fact, if $R\subset S$ is an extension and $T\in]R,S[$, we have  obviously $\mathrm{MSupp}(S/T)\cap\mathrm{MSupp}(T/R)=\emptyset$ if and only if $\mathrm{Supp}(S/T)\cap\mathrm{Supp}(T/R)=\emptyset$. That is the reason why we work only with maximal ideals. 
 
The following Proposition gives a first example of a split extension. 
 Section 6 is devoted to almost-Pr\"ufer extensions.
 
\begin{proposition}\label{split} \cite[Proposition 4.16]{Pic 5} and \cite[Proposition 3.6]{Pic 3} Let $R\subset S$ be an FCP extension. The following statements are  equivalent:
\begin{enumerate}
 \item $R\subset S$ is an almost-Pr\"ufer extension;
  \item $R\subset S$ splits at $\overline R$; 
  \item  $R\subset S$ splits at $\widetilde R$; 
 \item $R=\overline R\cap\widetilde R$ and $S=\overline R\widetilde R$;  \item There exists $U\in[R,S]$ such that $R=\overline R\cap U$ and $S=\overline RU$;
\end{enumerate}
If these conditions hold, $\widetilde R$ is the only $U\in[R,S]$ satisfying (5).
 \end{proposition}

\begin{definition} \label{split2}Let $R\subseteq S$ be a ring extension and $X\subseteq\mathrm{MSupp}_R(S/R)$. We say that an element $T$ of $[R,S]$ is a {\em splitter} of the extension at $X$ if $X=\mathrm{MSupp}_R(T/R)$ and $X^c=\mathrm{MSupp}_R(S/T)$. Clearly, such a splitter splits the extension and each element that splits an extension is a  splitter.
\end{definition}

\begin{remark} \label{split3} The FCP hypothesis on $R\subset S$ is necessary in Proposition~\ref{split}. In \cite[Remark 2.9(c)] {DPP2}, we give the following example. Let $(R,M)$ be a one-dimensional valuation domain with quotient field $S$, and set $T:=S[X]/(X^2)=S[x]$, where $x$ is the class of $X$ in $T$. We get that $ \widetilde R=S$ and $\overline R=R+Sx$. But neither $\widetilde R$ nor $\overline R$ are splitters since $\mathrm{MSupp}_R(T/\widetilde R)=\mathrm{MSupp}_R(\widetilde R/R)=\mathrm{MSupp}_R(T/\overline R)=\mathrm{MSupp}_R(\overline R/R)=\{M\}$. 

In fact, according to \cite[Theorem 4.6]{Pic 5}, a quasi-Pr\"ufer extension  $R\subset S$ is an almost-Pr\"ufer extension if and only if $\widetilde R$ is the complement of $\overline R$, and according to \cite[Corollary 4.7]{Pic 5}, if $R\subset S$ is an almost-Pr\"ufer extension, then $R\subset S$ splits at $\overline R$ if and only if $R\subset S$ splits at $\widetilde R$.
\end{remark}

According to Lemma~\ref{split1}, a splitter at $X$ is unique and we will denote it by $\sigma (X)$. It follows that if the extension splits at $T$, then $T= \sigma (X)$, where $X =\mathrm{MSupp}_R(T/R)$.
We observe that if $X=\mathrm{MSupp}(S/R)$, then $\sigma(\mathrm{MSupp}(S/R))=S$, and if $X=\emptyset$ then $\sigma(\emptyset)=R$. We say that $R$ and $S$ are {\em trivial} splitters.

The following Theorem shows that for an FCP $\mathcal B$-extension $R\subset S$ and for any $X\subseteq\mathrm{MSupp}(S/R)$, the splitter of the extension at $X$ always exists. In fact, it is a new characterization of $\mathcal B$-extensions. A precise study of splitters in such extensions is the subject of Section 5.

\begin{theorem}\label{1.15} Let $R\subset S$ be an FCP extension. Then, $R\subset S$ is a $\mathcal B$-extension if and only if for any $X\subseteq\mathrm{MSupp}(S/R)$, the splitter of the extension at $X$ exists. 
\end{theorem}

\begin{proof} Let $X\subseteq\mathrm{MSupp}(S/R)$ and set $Y:=X^c$. Assume that $R\subset S$ is a $\mathcal B$-extension, so that the map $\varphi:[R,S]\to\prod_{M\in\mathrm{MSupp}(S/R)}[R_M,S_M]$ defined by $\varphi(T)=(T_M)_{M\in\mathrm{MSupp}(S/R)}$ is bijective. In particular, there exists a unique $T\in[R,S]$ such that $T_M=S_M$ for each $M\in X$ and $T_M=R_M$ for each $M\in Y$. This implies in particular that $\mathrm{MSupp}(T/R)=X$ and $\mathrm{MSupp}(S/T)=Y$. Since $X\cap Y=\emptyset$, it follows that $T$ is the splitter of $R\subset S$ at $X$.

Conversely, assume that the splitter of $R\subset S$ at $X$ exists for any $X\subseteq\mathrm{MSupp}(S/R)$ and assume that $R\subset S$ is not a $\mathcal B$-extension. Then there exists $P\in\mathrm{Supp}(S/R)$ which is contained in two maximal ideals $M_1,M_2$ of $R$. Obviously, $M_1,M_2\in\mathrm{MSupp}(S/R)$. Set $X:=\{M_1\}$ and $Y:=X^c$, so that $M_2\in Y$, and let $T:=\sigma(X)$. It follows that $\mathrm{MSupp}(T/R)=X$ and $\mathrm{MSupp}(S/T)=Y$. Let $\{R_i\}_{i=0}^n$ be a maximal chain of $[R,S]$ containing $T$, so that $T=R_k$ for some $k\in\mathbb N_{n-1}$. Since $P\in\mathrm{Supp}(S/R)$, there is some $i\in\mathbb N_{n-1}$ such that $P=\mathcal{C}(R_i,R_{i+1})\cap R$. Then, $P\in\mathrm{Supp}_R(R_{i+1}/R_i)$, which implies that $M_1,M_2\in\mathrm{MSupp}_R(R_{i+1}/R_i)$. If $i<k$, then $R_i,R_{i+1}\in[R,T]$, and $M_2\in\mathrm{MSupp}_R(T/R)$, a contradiction. A similar contradiction holds if $i\geq k$. To conclude, $R\subset S$ is  a $\mathcal B$-extension.
\end{proof}

\begin{corollary}\label{4.48} Let $R\subset S$ be an FCP extension. Assume that there exists $X\subseteq\mathrm{MSupp}(S/R)$ such that  the splitter of the extension at $X$ does not exist. Then, there exists some $P\in\mathrm{Supp}(S/R)$ contained in two distinct ideals of $\mathrm{MSupp}(S/R)$. 
\end{corollary}
\begin{proof} According to Theorem ~\ref{1.15}, there exists $X\subseteq\mathrm{MSupp}(S/R)$ such that  the splitter of the extension at $X$ does not exist if and only if $R\subset S$ is not a $\mathcal B$-extension, which is equivalent to there exists some $P\in\mathrm{Supp}(S/R)$ such that $R/P$ is not local by Proposition ~\ref{6.5}; that is, there exists some $P\in\mathrm{Supp}(S/R)$ contained in two distinct ideals of $\mathrm{MSupp}(S/R)$. 
\end{proof}
  
When an FCP extension $R\subset S$ splits at some $T\in[R,S]$, we get a bijection between $[R,S]$ and $[R,T]\times[T,S]$ which allows to get a general formula about cardinalities of set of intermediary rings in case of FIP extensions. 
  
\begin{theorem}\label{4.5} Let $R\subset S$ be an FCP extension  that splits at $T$ and set $X:=\mathrm{MSupp}(T/R)$  Then, the following statements hold:
\begin{enumerate} 
\item There is an order-isomorphism $\psi:[R,S]\to[R,T]\times[T,S]$ defined by $\psi(V):=(V\cap T,VT)$. 

\item $T$ has a unique complement $T^o$, which is the splitter $\sigma (X^c)$ and satisfies $T+T^o=TT^o$. 
\item $\mathrm{MSupp}(T/R)=\mathrm{MSupp}(S/T^o)$ and $\mathrm{MSupp}(S/T)=$
\noindent$\mathrm{MSupp}(T^o/R)$;

\item If  $R\subset S$  is an FIP extension, $|[R,S]|=|[R,T]||[T,S]|$ holds.
\end{enumerate}
 \end{theorem}
\begin{proof} (1) is \cite[Lemma 3.7]{Pic 3} for the bijective property. The condition about the order is obvious, considering the product order on $[R,T]\times[T,S]$.

(2) The existence of $T^o$ comes from the fact that there exists a unique $U\in[R,S]$ such that $\psi(U)=(R,S)$, so that $R=U\cap T\ (*)$ and $S=UT\ (**)$ shows that $U$ is the complement $T^o$ of $T$. Since $\mathrm{MSupp}(S/R)=\mathrm{MSupp}(T^o/R)\cup\mathrm{MSupp}(S/T^o)$, by $(*)$, we get $(T^o)_M=R_M$ for any $M\in X$ and $(**)$ gives $(T^o)_M=S_M$ for any $M\in Y:=X^c$. It follows that $T_M+(T^o)_M=T_M(T^o)_M=S_M$ for any $M\in\mathrm{Supp}(S/R)$ leads to $T+T^o=TT^o$. At last, $X=\mathrm{MSupp}(S/T^o)$ and $Y=\mathrm{MSupp}(T^o/R)$ shows that $R\subset S$ splits at $T^o$. In particular, $\mathrm{MSupp}(S/R)=\mathrm{MSupp}(T/R)\cup\mathrm{MSupp}(T^o/R)$ since $\mathrm{MSupp}(S/R)=X\cup Y$.

(3) We have just proved that $X=\mathrm{MSupp}(S/T^o)=\mathrm{MSupp}(T/R)$ and $Y=\mathrm{MSupp}(T^o/R)=\mathrm{MSupp}(S/T)$.

(4) comes from (1).
\end{proof}

\begin{corollary}\label{simplifiable}Let $R\subset S$ be an FCP ring extension  split at $T$. 
\begin{enumerate}
\item $UT=VT$ implies $U=V$ for any $U,V\in[R,T^o]$.
\item $U\cap T=V\cap T$ implies $U=V$ for any $U,V\in[T^o,S]$.
\end{enumerate}
\end{corollary}
\begin{proof} (1) $U,V\in[R,T^o]$ implies $U\cap T=V\cap T=R$ since $T^o\cap T=R$. Applying the isomorphism $\psi$ of Theorem \ref{4.5}, we get $U=V$.

(2) $U,V\in[T^o,S]$ implies $UT=VT=S$ since $T^o T=S$. The same isomorphism $\psi$ gives $U=V$.
\end{proof}

\begin{proposition}\label{1.144}Let $R\subset S$ be an FCP extension. Then, any splitter is trivial if $R$ is a local ring.
\end{proposition}
\begin{proof} As $|\mathrm{MSupp}(S/R)|=1$ since $R\neq S$, we see that $R$ and $S$ are the only splitters.
\end{proof}

We recall that a ring extension $R\subset S$ is called {\em pinched} at $T\in]R,S[$ if $[R,S]=[R,T]\cup[T,S]$. We will see in Section 7 that when an extension $R\subset S$ is pinched at $T\in]R,S[$, then $R$ and $S$ are the only splitters.

A simple situation is given by the following Crosswise Exchange Lemma  which allowed us to develop   the properties of split FCP extensions.

\begin{proposition}\label{split4} (CE) \cite[Lemma 2.7]{DPP2} Let $R\subset S$ and $S\subset T$ be  minimal extensions,  $M:=\mathcal{C}(R,S)$, $N: \mathcal{C}(S,T)$ and $P:=N\cap R$. Suppose also that $P\not\subseteq M$. Then there exists $S'\in [R,T]$ such that $R\subset S'$ is minimal of the same type as $S\subset T$; and $S'\subset T$ is  minimal of the same type as $R\subset S$. Moreover, for any such $S'$, we have $[R,T]=\{R,S,S',T\}$. 
\end{proposition}

Analysing the hypotheses and the statement, we see that $\mathrm{MSupp}(S/R)$ $=\{M\},\ \mathrm{MSupp}(T/S)=\{P\}$ so that $\mathrm{MSupp}(S/R)\cap\mathrm{MSupp}(T/S)=\emptyset,\ S'=S^o$ and $[R,T]\cong[R,S]\times[S,T]$ by the map $\psi$ defined in Theorem \ref{4.5}(1). 

\begin{remark}\label{1.168} (1) In \cite[Proposition 1.6, page 88]{KZ}, the equation $T+V=TV$ of Theorem \ref{4.5} holds for arbitrary $T,V\in[R,S]$, when $R\subset S$ is a Pr\"ufer extension. Moreover, in this case, if $T$ has a complement, this complement is unique by \cite[Theorem 7.11 and Remarks 7.12, p.132]{KZ}. In fact, these authors define a complement for an $R$-submodule $I$ of $S$ containing $R$ as an $R$-submodule $J$ of $S$ such that $I+J=S$ and $I\cap J=R$. Their notation $I^o$ denotes the polar $J$ of $I$ as an $R$-submodule of $S$ such that $I\cap J=R$ and every $R$-submodule $K$ of $S$ such that $I\cap K\subseteq R$ is contained in $J$, and when a complement of $I$ exists, it coincides with its polar. 
 In the following result, $T^o$ denotes the complement of $T$ as defined in Section 3 and coincide with the complement defined by \cite{KZ}.
 Proposition \ref{1.170} holds for an arbitrary FCP extension and is similar to \cite[Propositions 7.2 and 7.3, p.130]{KZ} which is satisfied for Pr\"ufer extensions.
 
(2) Let $R\subset S$ be an FCP extension that splits at $T$; so that, $T=\sigma(X)$ for some $X\subseteq \mathrm{MSupp}(S/R)$. It follows that for any $M\in\mathrm{MSupp}(S/R)$, either $T_M=S_M$, when $M\in X\ (*)$, or $T_M=R_M$ when $M\in X^c\ (**)$. Then, $V:=R_M$ is the only complement of $T_M$ in case $(*)$ and $V:=S_M$ is the only complement of $T_M$ in case $(**)$. In both cases, we recover that $(T_M)^o=(T^o)_M$. 
 \end{remark}
 
\begin{proposition}\label{1.170} Let $R\subset S$ be an FCP extension  that splits at $T$. Then, the following properties hold:
\begin{enumerate} 
\item $T^o$ is the largest $V\in[R,S]$ such that $V\cap T=R$;
\item  $T^o=\{x\in S\mid T\cap Rx\subseteq R\}$.
\end{enumerate}
 \end{proposition} 
\begin{proof} As in Theorem \ref{4.5}, we set $X:=\mathrm{MSupp}(T/R)$ and $Y:=X^c$. Then, $T_M=S_M$ for any $M\in X$ and $(T^o)_M=S_M$ for any $M\in Y$.

(1) Obviously, $T^o\cap T=R$. 
  Let $V\in[R,S]$ be such that $V\cap T=R$, so that $V_M\cap T_M=V_M=R_M\subseteq (T^o)_M$ for any $M\in X$. Moreover, $V_M\subseteq S_M=(T^o)_M$ for any $M\in Y$, giving $V\subseteq T^o$. Hence, $T^o$ is the largest $V\in[R,S]$ such that $V\cap T=R$.

(2) Set $W:=\{x\in S\mid T\cap Rx\subseteq R\}$ and let $x\in W$, so that $T\cap Rx\subseteq R\ (*)$. If $M\in X$, then $(*)$ gives $T_M\cap R_M(x/1)\subseteq R_M$. But $T_M=S_M$ yields $x/1\in R_M\subseteq (T^o)_M$. If $M\in Y$, then $(T^o)_M=S_M$, so that $x/1\in (T^o)_M$, whence $x\in T^o$. 

Conversely, if $x\in T^o$, then $Rx\subseteq T^o$, which implies that $T\cap Rx\subseteq T\cap T^o=R$.
\end{proof}

\begin{proposition}\label{1.169} Let $R\subset S$ be an FCP extension  that splits at $T$. Then, the following properties hold:
\begin{enumerate} 
\item The map $\psi_1:[R,T^o]\to[T,S]$ defined by $\psi_1(V):=VT$ is an  order-isomorphism.

\item The map $\theta:[R,T]\times[R,T^o]\to[R,S]$ defined by $\theta(V,W):=VW$ is an order-isomorphism.
\end{enumerate}
 \end{proposition} 
\begin{proof} Let $\psi:[R,S]\to[R,T]\times[T,S]$ be the bijection defined by $\psi(V):=(V\cap T,VT)$ as in Theorem ~\ref{4.5}.

(1) Let $V,W\in[R,T^o]$ be such that $\psi_1(V)=\psi_1(W)=VT=WT$. Since $V\cap T=W\cap T=R$ because $T\cap T^o=R$, we get $\psi(V)=\psi(W)$, so that $V=W$ and $\psi_1$ is injective. 

Let $W\in[T,S]$. There exists $V\in[R,S]$ such that $\psi(V)=(R,W)$, so that $R=V\cap T$ and $W=VT$. Let $M\in\mathrm{Supp}(S/R)=\mathrm{Supp}(T/R)\cup\mathrm{Supp}(T^o/R)$ by Theorem ~\ref{4.5}. Then, $R_M=V_M\cap T_M$. If $M\in\mathrm{Supp}(T/R)$, then $T_M=S_M$ and $R_M=(T^o)_M$ lead to $V_M=(T^o)_M$. If $M\in\mathrm{Supp}(T^o/R)$, then $V_M\subseteq S_M=(T^o)_M$. It follows that $V\subseteq T^o$, giving $V\in[R,T^o]$ and $\psi_1$ is surjective, hence bijective. The property concerning the order is obvious. 

(2) Using (1) and localizing at any $M\in\mathrm{MSupp}(S/R)$, we may remark that $\psi\circ\theta={\rm Id}\times\psi _1$. Since $\psi$ and ${\rm Id}\times\psi _1$ are order-isomorphisms, so is $\theta$.
\end{proof}

\begin{remark}\label{4.7} We will see in Corollary ~\ref{4.6} that Theorem~\ref{4.5} has a converse when taking $T:=\overline R$ for an FCP extension $R\subset S$. But, for some $T\in[R,S]$, we may have a bijection $\psi:[R,S]\to[R,T]\times[T,S]$ defined by $\psi(U):=(U\cap T,UT)$ even if $\mathrm{MSupp}_R(S/T)\cap\mathrm{MSupp}_R(T/R)\neq\emptyset$.       

Consider the following situation. Let $R\subset T\subset S$ be a composite of two minimal extensions, where $(R,M)$ is a quasilocal ring, $R\subset T$ is decomposed and $T\subset S$ is ramified. Such a situation exists in the following example (see also \cite[Remark 2.9 (b)]{DPP2}). Take $R:=K$ a field, $T:=K^2$ and $S:=K\times(K[Z]/(Z^2))$, where $Z$ is an indeterminate, so that $R\subset T$ is minimal decomposed with $(R:S)=(0)$ and $T\subset S$ is minimal ramified with $(T:S)=K\times(0)\in\mathrm{Max}(T)$. In view of \cite[Lemma 2.8]{DPP2}, there exists $T'\in[R,S]$ such that $R\subset T'$ is minimal ramified and $T'\subset S$ is minimal decomposed. Since $|[R,S]|\geq 4$ and $R\subset S$ is of length 2 because infra-integral 
 (\cite[Lemma 5.4]{DPP2}), it follows from \cite[Theorem 6.1(5)]{Pic 6} that $[R,S]=\{R,T,T',S\}$. In fact, $T'=K[Z]/(Z^2)$. We have the following commutative diagram:
\centerline{$\begin{matrix}
 {} &      {}       & T' &       {}       & {} \\
 {} &\nearrow & {}  & \searrow & {} \\
 R &       {}     & {}  &      {}        & S \\
 {} &\searrow & {}  & \nearrow & {} \\
 {} &     {}       & T  &      {}        & {}     
 \end{matrix}$}
It follows that $|[R,S]|=4=|[R,T]||[T,S]|$ and $\psi:[R,S]\to[R,T]\times[T,S]$ defined by $\psi(U):=(U\cap T,UT)$ is a bijection. Indeed, we have $\psi(R):=(R,T),\ \psi(T):=(T,T),\ \psi(T'):=(R,S)$ and $\psi(S):=(T,S)$. To end, $\mathrm{MSupp}_R(S/T)\cap\mathrm{MSupp}_R(T/R)=(0)\neq\emptyset$.       
\end{remark}

\section{Splitters in $\mathcal B$-extensions}

Theorem \ref{1.15} shows that the splitter of $X$ exists for any $X\subseteq\mathrm{MSupp}(S/R)$ if and only if $R\subset S$ is  a $\mathcal B$-extension. We now examine in this section $\mathcal B$-extensions.

  Theorem \ref{4.5} has a converse in the Pr\"ufer case.

\begin{proposition}\label{1.171} 
Let $R\subset S$ be an FCP Pr\"ufer extension and $T\in[R,S]$ be such that $T^o$ exists. Then, there exists $X\subseteq\mathrm{MSupp}(S/R)$ such that $T=\sigma(X)$. 
\end{proposition}
\begin{proof} 
Recall that for a Pr\"ufer extension $R\subset S$ and any $T\in[R,S]$,  if $T$ has a complement, this complement is unique and is denoted by 
 $T^o$ (see Remark \ref{1.168}).

We have $\mathrm{MSupp}(S/R)=\mathrm{MSupp}(T/R)\cup\mathrm{MSupp}(S/T)$. Assume that $\mathrm{MSupp}(T/R)\cap\mathrm{MSupp}(S/T)\neq\emptyset$, and let $M\in\mathrm{MSupp}(T/R)\cap\mathrm{MSupp}(S/T)$, so that $T_M\neq R_M,S_M$. Since  $R=T\cap T^o$ and $S=TT^o$, this gives $R_M=T_M\cap (T^o)_M\ (*)$ and $S_M=T_M(T^o)_M\ (**)$, with $T_M\neq R_M,S_M$. But $R_M\subset S_M$ is chained according to \cite[Theorem 6.10]{DPP2}, so that $T_M$ and $(T^o)_M$ are comparable. Assume that $T_M\subseteq(T^o)_M$. Then $(*)$ gives $R_M=T_M$, a contradiction. If $(T^o)_M\subset T_M$, then $(**)$ gives $S_M=T_M$, a contradiction. To conclude, $\mathrm{MSupp}(T/R)\cap\mathrm{MSupp}(S/T)= \emptyset$, and $T=\sigma(X)$, where $X:=\mathrm{MSupp}(T/R)$. 
\end{proof}

 \begin{proposition}\label{6.50} Let $R\subset S$ be an FCP extension.
 \begin{enumerate}
 \item If the extension is chained, it is crucial. 
  \item Conversely, if the extension  is Pr\"ufer and  crucial, it is chained. 
 \end{enumerate}
\end{proposition}
\begin{proof} (1) Assume that $|\mathrm{MSupp}(S/R)|>1$ and let $M,M'\in\mathrm{MSupp}(S/R)$, $M\neq M'$. According to \cite[Lemma 1.8]{Pic 6}, there exist $T,U\in[R,S]$ such that $R\subset T$ and $R\subset U$ are minimal with $\mathcal{C}(R,T)=M$ and $\mathcal{C}(R,U)=M'$, so that $R\subset S$ is not chained; then (1) is proved.

(2) Assume that $R\subset S$ is Pr\"ufer and $|\mathrm{MSupp}(S/R)|=1$. Let $P\in\mathrm{Supp}(S/R)$ and $M\in \mathrm{V}( P)\cap \mathrm{Max}(R)$, so that $M\in \mathrm{MSupp}(S/R)$. Then, there is a unique maximal ideal $M$ of $R$ containing $P$. It follows from Proposition \ref{6.5} that $R\subset S$ is a $\mathcal B$-extension. In particular, there is an order-isomorphism $[R,S]\to [R_M,S_M]$. But $R_M$ is a local ring and $R_M\subset S_M$ is Pr\"ufer, whence chained by \cite[Theorem 6.10]{DPP2}, and so is $R\subset S$. 
\end{proof}

\begin{remark}\label{1.172} (1) We cannot extend Proposition \ref{1.171} to an arbitrary extension. Take for instance an infra-integral $M$-crucial extension $R\subset S$ such that ${}_S^+R\neq R,S$ and $M:=(R:S)$. Then, \cite[Theorem 6.1(5)]{Pic 6} shows that $R\subset S$ is a length 2 extension with 4 elements: $[R,S]=\{R,{}_S^+R,U,S\}$, where $U=({}_S^+R)^o$ while $M\in\mathrm{MSupp}({}_S^+R/R)\cap\mathrm{MSupp}(S/{}_S^+R)$.
 
 (2) Moreover, for an arbitrary FCP extension $R\subset S$ with $T,U\in]R,S[$ such that $U$ is a complement of $T$, this complement is not always unique. Consider for instance the following situation (\cite[Proposition 7.1]{DPPS}): $R\subset T$ is a minimal inert extension, $R\subset U$ is a minimal decomposed extension, and assume that the composite $S:=TU$ exists. Then, $R\subset T\subset S$ is a maximal chain and there exists a maximal finite chain $\{R,U,V,S\}$ of length 3. Then, $T$ is the complement of both $U$ and $V$.
   \end{remark}

\begin{corollary}\label{1.158} Let $R\subset S$ be an FCP $\mathcal B$-extension. If $T$ is a splitter, then $T=\sigma(\mathrm{MSupp}(T/R))$ and if $X\subseteq\mathrm{MSupp}(S/R)$, then $X=\mathrm{MSupp}(\sigma(X)/R)$.
 \end{corollary}
\begin{proof} Let $T$ be a splitter, so that there exists some $X\subseteq\mathrm{MSupp}(S/R)$ such that $\mathrm{MSupp}(T/R)=:X$, and then $T=\sigma(X)=\sigma(\mathrm{MSupp}(T/R))$.
 
Now, let $X\subseteq\mathrm{MSupp}(S/R)$ and set $T:=\sigma(X)$. By definition of the splitter at $X$, we have $\mathrm{MSupp}(T/R)=X=\mathrm{MSupp}(\sigma(X)/R)$.
 \end{proof}
 
 \begin{corollary}\label{1.159} Let $R\subset S$ be an FCP $\mathcal B$-extension. Then, for any $X\subseteq\mathrm{MSupp}(S/R)$, the following statements  hold:
\begin{enumerate}
\item $[R,\sigma(X)]=\{U\in[R,S]\mid \mathrm{MSupp}(U/R)\subseteq X\}$;
\item $[\sigma(X),S]=\{U\in[R,S]\mid \mathrm{MSupp}(S/U)\cap X=\emptyset\}$.
\end{enumerate}
 \end{corollary}
\begin{proof} Set $T:=\sigma(X)$. For any $M\in X$, we have $T_M=S_M$ and for any $M\not\in X$, we have $T_M=R_M$. 

(1) Let $U\in[R,T]$. Then, $U_M=R_M$ for any $M\not\in X$, so that $\mathrm{MSupp}(U/R)\subseteq X$. Conversely, assume that $\mathrm{MSupp}(U/R)\subseteq X$ for some $U\in[R,S]$. Then, for any $M\not\in X$, we have $M\not\in \mathrm{MSupp}(U/R)$, so that $T_M=R_M=U_M$. Since $T_M=S_M$ for any $M\in X$, we have $U_M\subseteq T_M$, so that $[R,\sigma(X)]=\{U\in[R,S]\mid \mathrm{MSupp}(U/R)\subseteq X\}$.
 
(2) Let $U\in[R,S]$. Since $U_M=S_M$ for any $M\in X$ if and only if $\mathrm{MSupp}(S/U)\cap X=\emptyset$, we get that $U_M=T_M$ for any $M\in X$ if and only if $\mathrm{MSupp}(S/U)\cap X=\emptyset$. But, $R_M=T_M\subseteq U_M$ for any $M\not\in X$. It follows that $U\in[T,S]$ if and only if $\mathrm{MSupp}(S/U)\cap X=\emptyset$, so that $[\sigma(X),S]=\{U\in[R,S]\mid \mathrm{MSupp}(S/U)\cap X=\emptyset\}$. 
 \end{proof}
 
 Let $R\subset S$ be an FCP $\mathcal B$-extension.
If $X=\{M\}$ for some $M\in\mathrm{MSupp}(S/R)$, we say that $\sigma(X)$ is an {\em elementary splitter} and write $\sigma(M):= \sigma(X)$. For such a splitter,  $R\subset \sigma(M)$ is $M$-crucial.  
 
If $P\in\mathrm{Supp}(S/R)\setminus\mathrm{Max}(R)$, we cannot define an elementary splitter associated to $P$, because there does not exists some $T\in[R,S]$ such that $T=\sigma(P)$, that is $\mathrm{Supp}(T/R)=\{P\}$. In fact, for any $M\in\mathrm{Max}(R)\cap \mathrm{V}(P)$, we have $M\in\mathrm{Supp}(T/R)$, a contradiction.
 
 In Example \ref{6.10}, we proved that $\mathrm{MSupp}_T(K/T)=\{M_1,M_2\}=\mathrm{MSupp}_T(K/ T_Q)$, so that there does not exist any $U\in[T,K]$ which is the splitter of $T\subset K$ at any $\{M_i\}$ for $i\in\{1,2\}$. The reason why is that $T\subseteq K$ is not a $\mathcal B$-extension.  
 
Corollary \ref{1.159} has the following application to elementary splitters.  

\begin{corollary}\label{1.175} Let $R\subset S$ be an FCP $\mathcal B$-extension and $M\in\mathrm{MSupp}(S/R)$. Then, $]R,\sigma(M)]=\{U\in[R,S]\mid R\subset U$ is $M$-crucial$\}$ and $\sigma(M)$ is a minimal element in the set of non trivial splitters of $R\subset S$.
\end{corollary}
\begin{proof} Set $T:=\sigma(M)$. 
 According to Corollary \ref{1.159}, we get that $[R,T]=\{U\in[R,S]\mid \mathrm{MSupp}(U/R)\subseteq\{M\}\}$. It follows that $U\in]R,T]\Leftrightarrow\mathrm{MSupp}(U/R)=\{M\}\Leftrightarrow R\subset U$ is $M$-crucial.

Assume there exists some $\emptyset\neq X\subseteq \mathrm{MSupp}(S/R)$ such that $R\subset V:=\sigma(X)\subseteq T$. Then, the previous result shows that $R\subset V$ is $M$-crucial, so that $X=\{M\}$ and $V=T$. 
\end{proof}

For an extension $R\subset S$ and $\Sigma$ a multiplicatively closed subset of $R$, Knebusch and Zhang define $R_{[\Sigma]}:=\{x\in S\mid sx\in R$ for some $s\in\Sigma\}$ \cite[Definition 10, page 18]{KZ}, a definition reminiscent  of large quotient  rings. In case $\Sigma=R\setminus P$ for some $P\in\mathrm{Spec}(R)$, then $R_{[\Sigma]}$ is replaced with $R_{[P]}$. In this case, $R_{[P]}$ is the pullback of the ring morphisms  $S\to S_P$ and $R_P\hookrightarrow  S_P$. In particular,  $R_{[P]}=S$ if and only if $P\not\in\mathrm{Supp}(S/R)$.  
  If $T\in[R,S[$, we write $R_{[P,T]}:=\{x\in T\mid sx\in R$ for some $s\in R\setminus P\}$ and drop $T$ when $T=S$.
 
We may compare the following Proposition with \cite[Theorem 1.8, page 4]{KK}, saying that, for a Pr\"ufer extension $R\subset S$ and some  $T\in[R,S]$ such that $T^o$ exists, the following holds: $T^o=\cap[R_{[M]}\, \,|\ M\in \Omega(T/R)]$, where $\Omega(T/R)$ is the set of maximal ideals $M$ of $R$ which are $T$-regular, that is such that $MT=T$. In fact, $\Omega(T/R)=\mathrm{Supp}(T/R)$ (see Proposition \ref{1.173}). 

\begin{proposition}\label{1.154} Let $R\subset S$ be an  FCP extension such that the splitter at $X$exists for some $X\subseteq\mathrm{MSupp}(S/R)$ and set  $\Sigma:=R\setminus(\cup[M\in X])$. The following results hold:
\begin{enumerate}
\item $R_{[\Sigma]}=\sigma(X^c)$ and is the (unique) complement of $\sigma(X)$;
\item $\sigma(X)=\cap[R_{[M]}\, \,|\ M\in X^c]$.
\end{enumerate}

\hskip -0,5cm Assume in addition that $R\subset S$ is integral with conductor $I$ and let    $M\in\mathrm{MSupp}(S/R)$.

$\mathrm{(3)}$ $R_{[M]}/R$ is an $(M/I)$-primary submodule of $S/R$ and $\sigma(M)/R=\cap[R_{[M']}/R\, \,|\ M'\in\mathrm{MSupp}(S/R),\ M'\neq M]$ is the reduced primary decomposition of $\sigma(M)/R$ into primary $R/I$-submodules of $S/R$.
\end{proposition}

\begin{proof} Set $T:=\sigma(X)$ which has a unique complement $T^ o=\sigma(X^c)$ by Theorem \ref{4.5}, $\Sigma':=R\setminus(\cup[M\in X^c]),\ V:=R_{[\Sigma]}$ and $V':=R_{[\Sigma']}$. Then, $T_M=S_M,\ (T^o)_{M'}=S_{M'},\ T_{M'}=R_{M'}$ and $(T^o)_M=R_M$ for any $M\in X$ and any $M'\in X^c$. 

(1) Let $x\in T^o$. Then, $x/1\in(T^o)_M=R_M$ for any $M\in X$, so that $x/1\in R_{\Sigma}$, and $x\in R_{[\Sigma]}$. Then, $T^o\subseteq R_{[\Sigma]}$. Conversely, let $x\in R_{[\Sigma]}$. Then, there exists $s\in\Sigma$ such that $sx\in R$. In particular, $s\not\in M$ for any $M\in X$, so that $x/1\in R_M=(T^o)_M$ for any $M\in X$. As $(T^o)_M=S_M$  for any $M\in X^c$, we get that $x/1\in(T^o)_M$ for any $M\in\mathrm{MSupp}(S/R)$, and then $x\in T^o$. To conclude $T^o=R_{[\Sigma]}$.

(2) By (1), we get $R_{[\Sigma']}= \sigma(X)$. Now, let $x\in S$. Then, $x\in R_{[\Sigma']}\Leftrightarrow$ there exists $s\in \Sigma'$ such that $sx\in R\Leftrightarrow$ there exists $s\not\in M$  for any $M\in X^c$ such that $sx\in R\Leftrightarrow x\in R_{[M]}$ for any $M\in X^c\Leftrightarrow x\in \cap[R_{[M]}\mid M\in X^c]$. Then, $R_{[\Sigma']}=\cap[R_{[M]}\, \,|\ M\in X^c]$. 

(3) Set $I:=(R:S)$ and $R':=R/I$. Since $R\subset S$ is an integral FCP extension, 
 $R\subset S$ is a $\mathcal B$-extension.  
It follows from \cite[Theorem 4.2]{DPP2} that $R'$ is an Artinian ring, and, in particular, a Noetherian ring. 
Moreover, $S/R$ is an $R'$-module as $S/R_{[M]},\ R_{[M]}/R$ and $\sigma(M)/R$.
According to \cite[Corollaire 1 of Proposition 7, page 136]{Bki A2}, we get that $\mathrm{Ass}(S/R_{[M]})=\mathrm{MSupp}(S/R_{[M]})=\{M\}$, where $\mathrm{Ass}(S/R_{[M]})$ is the set of associated prime ideals of $S/R_{[M]}$ (see \cite[D\'efinition 1, page 131]{Bki A2}). Then, $R_{[M]}/R$ is an $M/I$-primary submodule of $S/R$ by \cite[ Proposition 1 and D\'efinition 1, page 139]{Bki A2}.

Hence, the equality $\sigma(M)/R=\cap[R_{[M']}/R\, \,|\ M'\in\mathrm{MSupp}(S/R),\ M'$

\noindent $\neq M]$ of (2) is a reduced primary decomposition into distinct primary $R/I$-submodules of $S/R$.
\end{proof}

The following Proposition generalizes a result gotten for FCP Pr\"ufer extensions \cite[Remark 6.14 (b)]{DPP2}. If $R\subset S$ is Pr\"ufer, $\mathrm{Supp}(S/R)$ can be characterized as follows.

\begin{proposition}\label{1.173}   Let $R\subset S$ be a Pr\"ufer extension and $P\in\mathrm{Spec}(R)$. Then $P\in\mathrm{Supp}(S/R)$ if and only if $PS=S$, in which case  $R_{[P]}$ is a valuation subring of $S$.
If $P\not\in\mathrm{Supp}(S/R)$, then  $R_{[P]}=S$.
 \end{proposition}
\begin{proof}
Let $P\in\mathrm{Spec}(R)$. Since $R\subset S$ is a Pr\"ufer extension, it is a flat epimorphism (equivalently weakly surjective by \cite[Theorem 4.4, page 42]{KZ}). Then, $P\in\mathrm{Supp}(S/R)\Leftrightarrow R_P\neq S_P\Leftrightarrow PS=S$. If these conditions are satisfied, then $R_{[P]}$ is a valuation subring of $S$ \cite[Proposition 5.1, p.46]{KZ}.
\end{proof}

\begin{corollary}\label{1.174} Let $R\subset S$ be a ring extension. Then, the following statements hold:
\begin{enumerate}
\item  $R$ is an intersection of valuation subrings of $\widetilde R$.
\item Assume in addition that $R\subset S$ has FCP. Then $\overline R$ is an intersection of valuation subrings of $S$.
\item Assume in addition that $R\subset S$ is a Pr\"ufer extension. Then any splitter  is an intersection of valuation subrings of $S$.
\end{enumerate}
 \end{corollary}
\begin{proof} (1) comes from \cite[Remark 5.5, p.50 and Proposition 5.1, p.46]{KZ} since $R\subseteq\widetilde R$ is Pr\"ufer. Even if $R=\widetilde R$, $R\subseteq \widetilde R$ is Pr\"ufer.

(2) If $R\subset S$ has FCP, then $\overline R\subseteq S$ is Pr\"ufer. The result comes from the previous reference. 

(3) If $R\subset S$ is a Pr\"ufer extension, then any splitter is an intersection of valuation subrings of $S$ by Propositions \ref{1.154} and \ref{1.173}.
\end{proof}

In \cite{Pic 10}, we introduce the following sets. If $R\subseteq S$ has FCP, we set $\mathcal T:=\{T\in[R,S]\mid R\subset T$ crucial$\}$ and ${\mathcal T}_ M:=\{T\in\mathcal T\mid R\subset T\ M$-crucial$\}$ 
for some $M\in \mathrm{MSupp}(S/R)$.
 The sets ${\mathcal T}_ M$ give a partition of $\mathcal T$ associated to the equivalence relation $\mathcal R$ on $\mathcal T$, defined by $T\ \mathcal R\ T'$ if and only if $\mathrm{Supp}(T/R)=\mathrm{Supp}(T'/R)$. The next result improves \cite[Proposition 2.16]{Pic 10}, where we get that, in case $R\subseteq S$ has FIP and for any $M\in\mathrm{MSupp}(S/R)$, the set ${\mathcal T}_M$ has a greatest element $s(M):=\prod_{T\in\mathcal T_M}T$. In fact, $s(M)$ is a splitter, when $R\subset S$ is an FCP $\mathcal B$-extension. 
   
\begin{proposition}\label{1.151} If $R\subset S$ is an FCP $\mathcal B$-extension and $M\in\mathrm{MSupp}(S/R)$, then $s(M)=\sigma(M)$; so that, $]R,\sigma(M)]=\mathcal T_M$.
\end{proposition} 
\begin{proof} 

Set $U:=\sigma(M)$. By definition of the splitter at $M$, we get $\mathrm{MSupp}(U/R)=\{M\}$ and $\mathrm{MSupp}(S/U)=\mathrm{MSupp}(S/R)\setminus\{ M\}$. Then, $U\in\mathcal T_M$, so that $U\subseteq s(M)$. Moreover, $U_M=S_M$ and $U_{M'}=R_{M'}$ for any $M'\in\mathrm{MSupp}(S/R)\setminus\{M\}$. Let $T\in\mathcal T_M$. It follows that $\mathrm{Supp}(T/R)=\{M\}$, from which we infer that $T_{M'}=R_{M'}=U_{M'}$ for any $M'\in\mathrm{MSupp}(S/R)\setminus\{M\}$ and $T_M\subseteq S_M=U_M$. Then, $T\subseteq U$, which implies $s(M)\subseteq U$. To conclude, $U=s(M)$ and $\mathcal T_M\subseteq\ ]R,U]$. Obviously, any $T\in\ ]R,U]$ satisfies $\emptyset \neq \mathrm{MSupp}(T/R)\subseteq \mathrm{MSupp}(U/R)=\{M\}$, yielding $\mathrm{MSupp}(T/R)=\{M\}$. Therefore, $T\in\mathcal T_M$, and the last equality holds.
\end{proof}

\begin{proposition}\label{1.155} Let $R\subset S$ be an FCP $\mathcal B$-extension. Any splitter is a product, in a unique way, of elementary splitters of $[R,S]$:  for any $X\subseteq \mathrm{MSupp}(S/R)$, we have $\sigma(X)=\prod_{M\in X}\sigma(M)$. In particular, 
$S=\prod_{M\in\mathrm{MSupp}(S/R)}\sigma(M)$. Moreover, 
 for $M\in\mathrm{MSupp}(S/R)$, there exists $X\subseteq \mathrm{MSupp}(S/R)$ such that $\sigma(M)\subseteq\sigma(X)$ if and only if  $M\in X$.
\end{proposition} 
\begin{proof} Let $X:=\{M_1,\ldots,M_n\}\subseteq\mathrm{MSupp}(S/R)$. We may assume that $n>1$, because for $n=1$, the  result is obvious. Setting $T:=\sigma(X)$, we have, by definition of the splitter, $\mathrm{MSupp}(T/R)=X$ and $\mathrm{MSupp}(S/T)=\mathrm{MSupp}(S/R)\setminus X$. For any $M_i\in X$, set $T_i:=\sigma(M_i)$ and $U:=\prod_{i=1}^nT_i$. 

Let $M\in\mathrm{MSupp}(S/R)\setminus X=\mathrm{MSupp}(S/T)$. Then, $T_M=R_M=(T_i)_M$, for any $i\in\mathbb N_n$, so that $T_M=U_M$.

If $M\in X$, there exists some $i\in\mathbb N_n$ such that $M=M_i$. Then, $T_M=S_M$ and $(T_i)_M=S_M$. Since $U_M=(\prod_{j=1,j\neq i}^n(T_j)_M)(T_i)_M=S_M$, we deduce that $T_M=U_M$. 

To conclude, $T=U$. The uniqueness of the product is obvious, considering the relation $\sigma(X)=\prod_{M\in X}\sigma(M)$, since any different product of elementary splitters would lead to a different support, and then to a different splitter.

As $S$ is a splitter by Definition \ref{split2}, the last equality is obvious.

Now, assume that an elementary splitter $U:=\sigma(M)$ is contained in a splitter $T:=\sigma(X)$. We have $\mathrm{MSupp}(T/R)=X$ and $\mathrm{MSupp}(U/R)=\{M\}$, with $U\subseteq T$. This implies $\mathrm{MSupp}(U/R)\subseteq \mathrm{MSupp}(T/R)$ and then, $M\in X$.
Conversely, $M\in X$ implies $\sigma(M)\subseteq\sigma(X)$ by the previous equality.
\end{proof}

\begin{corollary}\label{1.156} Let $R\subset S$ be an FCP $\mathcal B$-extension and  $X,Y\subseteq \mathrm{MSupp}(S/R)$. The following properties hold:
\begin{enumerate}
\item $X\subseteq Y\Leftrightarrow \sigma(X)\subseteq \sigma(Y)$;
\item $\sigma(X\cap Y)= \sigma(X)\cap \sigma(Y)$;
\item $\sigma(X\cup Y)= \sigma(X) \sigma(Y)$;
\item If $M,M'\in\mathrm{MSupp}(S/R),\ M\neq M'$, then $\sigma(M)\cap \sigma(M')=R$.
\end{enumerate}
\end{corollary} 
\begin{proof} Set $X\cap Y:=\{M_1,\ldots,M_k\},\ X:=\{M_1,\ldots,M_k,M_{k+1},\ldots,M_n\}$ and $Y:=\{M_1,\ldots,M_k,M'_{k+1},\ldots,M'_m\}$. 

(1) Assume that $X\subseteq Y$, so that $n=k$. According to Proposition \ref{1.155}, we get $\sigma(X)=\prod_{i=1}^k\sigma(M_i)\subseteq \prod_{i=1}^k\sigma(M_i)\prod_{j=k+1}^m\sigma(M'_j)=\sigma(Y)$.

Conversely, assume that $\sigma(X)\subseteq\sigma(Y)$. Then, 
$\mathrm{MSupp}(\sigma(X)/R)\subseteq\mathrm{MSupp}(\sigma(Y)/R)$, so that $X\subseteq Y$ by Corollary \ref{1.158}. 

(2) By (1), we get $\sigma(X\cap Y)\subseteq\sigma(X)\cap\sigma(Y)$. Set $T:=\sigma(X)\cap\sigma(Y)$ and $Z:=\mathrm{MSupp}(T/R)$. Since $T\subseteq\sigma(X)$, we have $Z\subseteq \mathrm{MSupp}(\sigma(X)/R)=X$.
 In the same way, $Z\subseteq Y$, so that $Z\subseteq X\cap Y$. Assume that $Z\neq X\cap Y$. Then, let $M\in(X\cap Y)\setminus Z$. Since $M\not\in Z=\mathrm{MSupp}(T/R)$, we get $R_M=T_M$. But $M\in X\cap Y$ implies that $\sigma(M)\subseteq\sigma(X)\cap\sigma(Y)=T$, and then, $R_M\subset(\sigma(M))_M\subseteq T_M$, a contradiction. It follows that $\sigma(X\cap Y)= \sigma(X)\cap \sigma(Y)$.

(3) We use the notation of (1).

Then, $X\cup Y:=\{M_1,\ldots,M_k,M_{k+1},\ldots,M_n,M'_{k+1},\ldots,M'_m\}$ and $\sigma(X)\sigma(Y)=\prod_{i=1}^k\sigma(M_i)\prod_{j=k+1}^n\sigma(M_j)\prod_{j=k+1}^m\sigma(M'_j)=\sigma(X\cup Y)$. 

(4) Since $M\neq M'$, then $\{M\}\cap \{M'\}=\emptyset$. The results follows from (2) since $\sigma(\emptyset)=R$.
\end{proof}

Here is an example illustrating Proposition \ref{1.155}.

\begin{example}\label{1.166} \cite[Example 3.16]{Pic 11} Let $k\subset L$ be a radicial ({\em i.e.} purely inseparable) field extension of degree $p^2$ and $K$ be the only proper subalgebra of $L$. Set $R:=k^2,\ R_1:=[k[X]/(X^2)]\times k,\ R_2:=k\times K,\ R_3:=k\times L,\ S:=[k[X]/(X^2)]\times L,\ M:=0\times k$ and $N:=k\times 0$. Then, $\mathrm{Max}(R)=\{M,N\}$ with $M\neq N$.  
 We proved that $[R,S]=\{R,R_1,R_2,R_1R_2,R_3,S\}$ with the following commutative diagram:
 
  \centerline{$\begin{matrix}
 {} &        {}      & R_1 &       {}      &{}             & {}            & {} \\
 {} & \nearrow & {}     & \searrow & {}            & {}            & {} \\
 R &       {}      & {}     &      {}        & R_1R_2 & {}            & {} \\
 {} & \searrow & {}     & \nearrow & {}            & \searrow & {} \\
 {} &      {}       & R_2 & \to           & R_3        & \to           & S
 \end{matrix}$}
  
Moreover, $R\subset R_1$ is a minimal extension with ${\mathcal C}(R,R_1)=M,\ R\subset R_2$ is a minimal extension with ${\mathcal C}(R,R_2)=N$ and $R_2\subset R_3$ is a minimal extension with ${\mathcal C}(R_2,R_3)=N$, so that $\mathrm{MSupp}(S/R)=\{M,N\},$ 

\noindent$\mathrm{MSupp}(R_1/R)=\{M\}$ and $\mathrm{MSupp}(R_3/R)=\{N\}$. It was also proved that $(R_1)_M=S_M$ and $S_N=(R_3)_N$, so that $\mathrm{MSupp}(S/R_1)=\{N\}$ and $\mathrm{MSupp}(S/R_3)=\{M\}$. Then, $R_1=\sigma(M)$ and $R_3=\sigma(N)$, with $R_3$ the complement of $R_1$. We recover the relation of Proposition \ref{1.155}: $\sigma(M)\sigma(N)=S$.
\end{example} 

The following lemma is need in the sequel. 
\begin{lemma}\label{Q2} Let $R\subset S$ be an $M$-crucial integral FCP extension and $I:=(R:S)$. There exists some positive integer $n$ such that $M^n\subseteq I$.
\end{lemma}

\begin{proof} Since $R\subset S$ is an integral FCP extension, $R/I$ is an Artinian ring by \cite[Theorem 4.2]{DPP2}. Moreover, since $R\subset S$ is an $M$-crucial extension, $\mathrm{MSupp}(S/R)=\{M\}$, so that $M$ is the only maximal ideal of $R$ containing $I$. Then, $R/I$ is an Artinian local ring whose maximal ideal is $M/I$; so that, there exists some positive integer $n$ such that $(M/I)^n=0$, whence $M^n\subseteq I$.
\end{proof}

\begin{corollary}\label{1.163} Let $R\subset S$ be an FCP $\mathcal B$-extension, $I$ an ideal shared by $R$ and $S$, $X\subseteq \mathrm{MSupp}(S/R)$, $\Sigma$ a multiplicatively closed subset of $R$ and $Y:=\{M\in X\mid M\cap\Sigma=\emptyset\}$. The following properties hold:
\begin{enumerate}

\item $R/I\subset S/I$ is an FCP $\mathcal B$-extension. 
\item Setting $R':=R/I,\ S':=S/I$ and $X':=\{M/I\mid M\in X\}$, we have $X'\subseteq \mathrm{MSupp}(S'/R')$ and $\sigma_{[R',S']}(X')=\sigma(X)/I$;
 
\item If $R\subset S$ is integral, then, 
$\sigma(X)_{\Sigma}=\sigma(Y)_{\Sigma}=\sigma_{[R_{\Sigma},S_{\Sigma}]}(Y')$ where $Y':=\{M_{\Sigma}\mid M\in Y\}$.

\hskip -2cm If $M\cap\Sigma=\emptyset$ for any $M\in X$, then (4) and (5) hold:
 
\item $R_{\Sigma}\subset S_{\Sigma}$ is an FCP $\mathcal B$-extension. 
\item Setting $X'':=\{M_{\Sigma}\mid M\in X\}$, we have $X''\subseteq \mathrm{MSupp}(S_{\Sigma}/R_{\Sigma})$ and $\sigma_{[R_{\Sigma},S_{\Sigma}]}(X'')=\sigma(X)_{\Sigma}$.
\end{enumerate}

\end{corollary} 
\begin{proof} (1)  comes from \cite[Proposition 3.7]{DPP2} for the FCP property and from Proposition \ref{6.5} for the $\mathcal B$-extension property. 

(2)  Since $I\subseteq (R:S)$, we get that $I\subseteq M$ for any $M\in\mathrm{MSupp}(S/R)$. Set $M':=M/I$ for such an $M$. Then, $R_M/I_M\cong R'_{M'}$ and $S_M/I_M\cong S'_{M'}$, so that $M'\in\mathrm{MSupp}(S'/R')$, giving $X'\subseteq\mathrm{MSupp}(S'/R')$. Moreover, Corollary \ref{1.159} shows that $[R',\sigma_{[R',S']}(X')]=\{T'\in[R',S']\mid$

\noindent$\mathrm{MSupp}(T'/R')\subseteq X'\}$, but any $T'\in[R',S']$ is of the form $T/I$ for some $T\in[R,S]$, so that 
 $\mathrm{MSupp}(T'/R')\subseteq X'$ is equivalent to $\mathrm{MSupp}(T/R)\subseteq X$, giving $\sigma_{[R',S']}(X')=\sigma(X)/I$. 

(3) Let $M\in X$ be such that $M\cap\Sigma\neq\emptyset$ and set $U:=\sigma(M)$. According to Lemma \ref{Q2}, there exists some positive integer $n$ such that $M^n\subseteq(R:U)$. Let $x\in M\cap\Sigma$. Then, $x^n\in(R:U)\cap\Sigma$, so that $R_{\Sigma}=U_{\Sigma}=\sigma(M)_{\Sigma}$. By Proposition \ref{1.155}, we have $\sigma(X)=\prod_{M\in X}\sigma(M)=\prod_{M\in Y}\sigma(M)\prod_{M\in X\setminus Y}\sigma(M)$, which leads to 
 $\sigma(X)_{\Sigma}=(\prod_{M\in Y}\sigma(M)_{\Sigma})(\prod_{M\in X\setminus Y}\sigma(M)_{\Sigma})=\prod_{M\in Y}\sigma(M)_{\Sigma}=\sigma(Y)_{\Sigma}$. 
 
(4) comes from \cite[Proposition 3.1]{DPP5} for the FCP property and from Proposition \ref{6.5} for the $\mathcal B$-extension property since the maximal ideals of $R_{\Sigma}$ are the $M_{\Sigma}$ for any $M\in\mathrm{Max}(R)$ such that $M\cap\Sigma=\emptyset$.

(5) Let $M_{\Sigma}\in X''$ for some $M\in X$. Then $R_M\neq S_M$. But $R_M\cong (R_{\Sigma})_{M_{\Sigma}}$ and $S_M\cong (S_{\Sigma})_{M_{\Sigma}}$ leads to $M_{\Sigma}\in\mathrm{MSupp}(S_{\Sigma}/R_{\Sigma})$, so that $X''\subseteq\mathrm{MSupp}(S_{\Sigma}/R_{\Sigma})$. As in (2), Corollary \ref{1.159} shows that $[R_{\Sigma},\sigma_{[R_{\Sigma},S_{\Sigma}]}(X'')]=\{T''\in[R_{\Sigma},S_{\Sigma}]\mid\mathrm{MSupp}(T''/R_{\Sigma})\subseteq X''\}$, but any $T''\in[R_{\Sigma},S_{\Sigma}]$ is of the form $T_{\Sigma}$ for some $T\in[R,S]$; so that,  $\mathrm{MSupp}(T_{\Sigma}/R_{\Sigma})\subseteq X''$ is equivalent to $\mathrm{MSupp}(T/R)\subseteq X$ and then 
$\sigma_{[R_{\Sigma},S_{\Sigma}]}(X'')=\sigma(X)_{\Sigma}$. 
\end{proof}

\begin{proposition}\label{1.157} Let $R\subset S$ be an FCP $\mathcal B$-extension and $T\in[R,S]$. Then, $\sigma_{[R,T]}(X)=\sigma_{[R,S]}(X)\cap T$ for any $X\subseteq \mathrm{MSupp}(T/R)$. \end{proposition} 

\begin{proof} Let $M\in\mathrm{MSupp}(T/R)$. 
Then, $R_M\neq T_M\subseteq S_M$, so that $M\in\mathrm{MSupp}(S/R)$.
 Set $U:=\sigma_{[R,S]}(M)$ and $V:=\sigma_{[R,T]}(M)$, which exists since $R\subseteq T$ is an FCP $\mathcal B$-extension by Proposition \ref{6.8} (3).
Then, $R_M\neq V_M=T_M\subseteq S_M=U_M$, which gives $V_M=T_M\cap U_M$. Let $M'\in \mathrm{MSupp}(S/R),\ M'\neq M$. Then, $R_{M'}=V_{M'}=U_{M'}$ gives $V_{M'}=T_{M'}\cap U_{M'}$. It follows that $\sigma_{[R,T]}(M)=\sigma_{[R,S]}(M)\cap T\ (*)$.  
 
Now consider  $X\subseteq\mathrm{MSupp}(T/R)$, a subset of $\mathrm{MSupp}(S/R)$. According to Proposition \ref{1.155}, we have $\sigma(X)=\prod_{M\in X}\sigma(M)$. We also have $\sigma_{[R,T]}(X)
=\prod_{M\in X}\sigma_{[R,T]}(M)$, which implies by the first part 
 $(\sigma_{[R,T]}(X))_M$
 $=(\sigma_{[R,T]}(M))_M=(\sigma_{[R,S]}(M)\cap T)_M=(\sigma_{[R,S]}(M))_M\cap T_M$ for any $M\in X$. But $(\sigma_{[R,S]}(X)\cap T)_M=(\sigma_{[R,S]}(X))_M\cap T_M=(\sigma_{[R,S]}(M))_M\cap T_M=(\sigma_{[R,T]}(X))_M$. 
    
Then, $(\sigma_{[R,T]}(X))_M=(\sigma_{[R,S]}(X)\cap T)_M$ for any $M\in X$. If $M\in\mathrm{MSupp}(S/R)\setminus X$, then $R_M=(\sigma_{[R,T]}(X))_M=(\sigma_{[R,S]}(X))_M=(\sigma_{[R,S]}(X))_M\cap T_M=(\sigma_{[R,S]}(X)\cap T)_M$. To conclude, $\sigma_{[R,T]}(X)=\sigma_{[R,S]}(X)\cap T$. 
 \end{proof}
 
\begin{proposition}\label{1.160} Let $R\subset S$ be an FCP $\mathcal B$-extension with $\mathrm{MSupp}(S/R)$
 $=:\{M_1,\ldots,M_n\}$ and $X_k:=\{M_1,\ldots,M_k\}$ for each $k\in\mathbb N_n$. 
  Assume $n>1$.
 Then, 
\begin{enumerate} 
\item $\sigma(X_{k+1})=\sigma(X_k)\sigma(M_{k+1})$ for each $k\in\mathbb N_{n-1}$;
\item $\{\sigma(X_k)\}_{k=1}^n$ is a chain in $[R,S]$ such that $\sigma(X_k)\subset\sigma(X_{k+1})$ is a $M_{k+1}\sigma(X_k)$-crucial extension, for any $k\in\mathbb N_{n-1}$, where $M_{k+1}\sigma(X_k)$ is the unique maximal ideal of $\sigma(X_k)$ lying over $M_{k+1}$;
\item There are $n!$ such chains, each chain corresponding to a permutation of $\{1,\ldots,n\}$.
\end{enumerate} 
\end{proposition} 
\begin{proof} (1) Let $k\in\mathbb N_{n-1}$. According to Proposition \ref{1.155}, we have $\sigma(X_{k+1})=\prod_{i=1}^{k+1}\sigma(M_i)=\sigma(M_{k+1})\prod_{i=1}^k\sigma(M_i)=\sigma(X_k)\sigma(M_{k+1})$.
  
(2) By (1), we get that $\sigma(X_k)\subseteq\sigma(X_{k+1})$. Moreover, $M_{k+1}\not\in X_k$. Then \cite[Lemma 2.4]{DPP2} yields that $M_{k+1}\sigma(X_k)$ is the unique maximal ideal of $\sigma(X_k)$ lying over $M_{k+1}$. It follows from the relations $M_{k+1}\in X_{k+1}=\mathrm{MSupp}(\sigma(X_{k+1})/R)=\mathrm{MSupp}(\sigma(X_k)/R)\cup\{M_{k+1}\}$ that $\{M_{k+1}\sigma(X_k)\}=\mathrm{MSupp}_{[\sigma(X_k),\sigma(X_{k+1})]}(\sigma(X_{k+1})/\sigma(X_k))$, giving that $\sigma(X_k)\subset\sigma(X_{k+1})$ is an $M_{k+1}\sigma(X_k)$-crucial extension. 
  
(3) To each permutation $\{i_1,\ldots,i_n\}$ of $\{1,\ldots,n\}$, we can build a unique chain $\{\sigma(X_{i_k})\}$. Then, there are $n!$ such chains, each chain corresponding to a permutation of $\{1,\ldots,n\}$.
  \end{proof}
 
 A ring extension $R\subseteq S$ is called {\em arithmetic} if it is locally chained.

 \begin{theorem}\label{1.161} Let $R\subset S$ be an FCP extension. The following conditions are equivalent:
\begin{enumerate}
\item Any element of $[R,S]$ is a splitter;

\item $R \subset S$ is  locally minimal;

\item $R\subset S$ is an FIP $\mathcal B$-extension such that $|[R,S]|=2^{|\mathrm{MSupp}(S/R)|}$;

\item $R \subset S$ is a Boolean and arithmetic FIP extension.
\end{enumerate}  
\end{theorem} 
\begin{proof} (1) $\Leftrightarrow\ \mathrm{MSupp}(T/R)\cap\mathrm{MSupp}(S/T)=\emptyset$ for any $T\in[R,S]$, by the definition of a splitter, which is equivalent to $R\subset S$ is a locally minimal extension by \cite[Theorem 2.27 ]{Pic 10}. 

(2) $\Leftrightarrow$ (3) by Corollary \ref{6.51}.

(2) +(3) $\Rightarrow$ (4) by \cite[Corollary 3.6 ]{Pic 10} and (4) $\Rightarrow$ (2) by the same reference.
  \end{proof}
 
\begin{corollary}\label{1.153} Let $R\subset S$ be an FCP $\mathcal B$-extension. Then, for any $M\in\mathrm{MSupp}(S/R)$, the map $\varphi_M:[R,\sigma(M)]\to[R_M,S_M]$ defined by $\varphi_M(T):=T_M$ is an order-isomorphism.
\end{corollary} 
\begin{proof} Let $M\in\mathrm{MSupp}(S/R)$. According to \cite[Proposition 2.22]{Pic 10} and Proposition \ref{1.151}, $\varphi_M$ is surjective and obviously a respecting order map. Let $T,T'\in[R,\sigma(M)]$ be such that $\varphi_M(T)=\varphi_M(T')$, so that $T_M=T'_M$. But, $T,T'\in [R,\sigma(M)]$ implies that $T_{M'}=R_{M'}=T'_{M'}$ for any $M'\in\mathrm{MSupp}(S/R),\ M'\neq M$. Then, $T=T'$, and $\varphi_M$ is bijective. It follows that $\varphi_M$  is an order-isomorphism.
\end{proof}

In light of Corollary \ref{1.153}, checking that an FCP $\mathcal B$-extension satisfies some type of property which is preserved by localization and globalization, the following proposition shows that we can limit to verify this property only for crucial extensions. We say that   an extension $R\subset S$ satisfies a {\em local-global} property
 $(\mathcal P)$ in case the extension satisfies $(\mathcal P)$ if and only if $R_M\subset S_M$ satisfies $(\mathcal P)$ for any $M\in\mathrm{MSupp}(S/R)$. 

\begin{proposition}\label{1.164} Let $(\mathcal P)$ be a local-global property.
Then, an FCP $\mathcal B$-extension $R\subset S$ satisfies $(\mathcal P)$ if and only if $R\subset\sigma(M)$ satisfies $(\mathcal P)$ for any $M\in\mathrm{MSupp}(S/R)$.
\end{proposition} 
\begin{proof} For each $M\in\mathrm{MSupp}(S/R)$, we have 
$\mathrm{MSupp}(\sigma(M)/R)=\{M\}$ by Corollary \ref{1.158}, so that $\sigma(M)_M=S_M$. Let $(\mathcal P)$ be a local-global property.
Then, $R\subset S$ satisfies $(\mathcal P)$ if and only if $R_M\subset S_M$ satisfies $(\mathcal P)$ for any $M\in\mathrm{MSupp}(S/R)$ if and only if $R_M\subset\sigma(M)_M$ satisfies $(\mathcal P)$ for any $M\in\mathrm{MSupp}(S/R)$  if and only if $R\subset\sigma(M)$ satisfies $(\mathcal P)$ for any $M\in\mathrm{MSupp}(S/R)$ since $\sigma(M)_M=S_M$ and $\sigma(M)_{M'}=R_{M'}$ for any $M'\in\mathrm{MSupp}(S/R),\ M'\neq M$.
\end{proof}

The next Proposition shows that some splitters are related to closures associated to a specified property of ring extensions. Let $(\mathcal{P})$ be a property concerning a class $\mathcal C$ of ring extensions, stable under subextensions (i.e. $R\subseteq S$ in $\mathcal C$ and $[U,V]\subseteq[R,S]$ imply $U\subseteq V$ in $\mathcal C)$.  

We say that $(\mathcal P)$ admits a closure in $\mathcal C$ if the following conditions (i), (ii) and (iii) hold for any extension $R\subset S$ in $\mathcal C$:

(i) For any tower of extensions $R\subseteq U\subseteq S$, then $R\subseteq S$ has $(\mathcal{P})$ if and only if $R\subseteq U$ and $U\subseteq S$ have $(\mathcal{P})$.

(ii) There exists a largest subextension $T\in[R,S]$ such that $R\subseteq T$ has $(\mathcal{P})$. 

(iii) No subextension $U\subseteq V$ of $T\subseteq S$ has $(\mathcal{P})$. 

According to \cite[the paragraph after Theorem 1.4]{Pic 4}, such a $T$ is unique, is called the $\mathcal{P}$-closure of $R$ in $S$ and is denoted by $R^{\mathcal{P}}$. 
Some instances are the separable closure in the class of algebraic field extensions, the seminormalization and the t-closure.

\begin{proposition}\label{1.165} Let $(\mathcal{P})$ be a local-global property of ring extensions admitting a $\mathcal P$-closure in a class $\mathcal C$ of ring extensions. Let $R\subset S$ be an FCP $\mathcal B$-extension which belongs to $\mathcal C$. Let $T:=R^{\mathcal{P}}$ be its $\mathcal{P}$-closure and set $X:=\{M\in\mathrm{MSupp}(S/R)\mid R_M\subset S_M$ satisfies $(\mathcal P)\}$. Then $\sigma(X)\subseteq T$, with equality if and only if $R\subset S$ splits at $T$. If this last condition holds, then $T^o$ is the least subextension $W\in[R,S]$ such that $W\subseteq S$ has $(\mathcal{P})$. 
\end{proposition} 
\begin{proof} Set $V:=\sigma(X)$ and $Y:=X^c$. For each $M\in X$, we have $V_M=S_M$, so that $R_M\subseteq S_M=V_M$ satisfies $(\mathcal P)$. For each $M\in Y$, we have $V_M=R_M$, so that $R_M\subseteq V_M$ satisfies $(\mathcal P)$. Then, $R\subseteq V$ satisfies $(\mathcal P)$, leading to $V\subseteq T$. 

If $V=T$, then $T=\sigma(X)$ is a splitter and $R\subset S$ splits at $T$. Conversely, assume that $R\subset S$ splits at $T$ and set $X':=\mathrm{MSupp}(T/R)$, so that $T=\sigma(X')$. If $V\subset T$, it follows that $X\subset X'$ by Corollary \ref{1.156}. Let $M\in X'\setminus X$. Then, $R_M=V_M$ and $T_M=S_M$, so that $R_M\subset S_M$ satisfies $(\mathcal P)$, a contradiction since $M\not\in X$. Then, we get  $V=T$. 

Assume that $T=V=\sigma(X)$. For $M\in X$, we have $(T^o)_M=R_M$ and $T_M=S_M$, so that $(T^o)_M\subseteq S_M$ satisfies $(\mathcal P)$ and for $M\in Y$, we have $(T^o)_M=S_M$. It follows that $T^o\subseteq S$ satisfies $(\mathcal P)$. At last, let $W\in[R,S]$ be such that $W\subseteq S$ has $(\mathcal{P})$. For $M\in X$, we have $(T^o)_M=R_M$, so that $(T^o)_M\subseteq W_M$. 

Assume that $TW\neq S$. Since $T,W\subseteq TW\subset S$, this implies that $TW\subset S$ has, at the same time, $(\mathcal P)$ because so has $W\subseteq S$ and has not $(\mathcal P)$ because so has not $T\subseteq S$, a contradiction. It follows that $TW=S$, leading to $T_MW_M=S_M=W_M=(T^o)_M$ for $M\in Y$. To conclude, $T^o\subseteq W$. Then, $T^o$ is the least subextension $W\in[R,S]$ such that $W\subseteq S$ has $(\mathcal{P})$.
\end{proof}

 We give more details about a result from \cite{DPP2}.

\begin{proposition}\label{1.167} Let $R\subset S$ be a seminormal infra-integral FCP $M$-crucial extension. Set $n:=\ell[R,S]$. Then, $M=(R:S),\ |\mathrm{V}_S(M)|=n+1$ and $S/M\cong (R/M)^{n+1}$.
 \end{proposition} 
\begin{proof} Since $R\subset S$ is an integral FCP $M$-crucial extension, it is a $\mathcal B$-extension, so that there is an order-isomorphism $[R,S]\cong[R_M,S_M]$. Then, $|\mathrm{V}_S(M)|=|\mathrm{Max}(S_M)|$ and $n=\ell[R_M,S_M]$. It follows that we can assume that $(R,M)$ is a local ring. 
According to Proposition \ref{1.91} and \cite[Lemma 5.4]{DPP2}, $(R:S)=M$ is an intersection of the maximal ideals of $S,\ |\mathrm{Max}(S)|=n+1$, and $S/M\cong(R/M)^{n+1}$, by the Chinese Remainder Theorem, since $R\subset S$ is infra-integral. 
\end{proof}

\section{ Integral closure and  Pr\"ufer hull as splitters} 

The following Theorem is a first application of Theorem ~\ref{4.5} to almost-Pr\"ufer extensions because an almost-Pr\"ufer FCP extension $R\subseteq S$ splits at $\overline R$ and $\widetilde R$.
By the way, we generalize a result gotten by Ayache in \cite[Theorem 20]{A} for extensions of integral domains.

\begin{theorem}\label{4.10} Let $R\subset S$ be an FCP extension. There is an order-isomorphism $\psi:[R,S]\to\{(T',T'')\in[R,\overline R]\times[\overline R,S]\mid T'\subseteq T''$ almost-Pr\"ufer$\}$ defined by $\psi(T):=(T\cap\overline R,\overline RT)$ for each $T\in[R,S]$. In particular, if $R\subset S$ has FIP, then $|[R,S]|\leq|[R,\overline R]||[\overline R,S]|$.  
\end{theorem}

\begin{proof}  Let $(T',T'')\in[R,\overline R]\times[\overline R,S]$. Then, $\overline R$ is also the integral closure $\overline T'$ of $T'$ in $T ''$ (and in $S$). 

Let $T\in[R,S]$ and set $T':=T\cap\overline R$ and $T'':=\overline RT$. Then $T\in[T',T'']$ and $(T',T'')\in[R,\overline R]\times[\overline R,S]$. Assume first that $T'=T''$; so that, $T'=T''=\overline R$, which implies that $T=\overline R$ and $\mathrm{Supp}_{T'}(\overline R/T')=\mathrm{Supp}_{T'}(T''/\overline R)=\emptyset$. It follows that $T'\subseteq T''$ is almost-Pr\"ufer. Assume now that $T'\neq T''$. Applying Proposition ~\ref{split} to the extension $T'\subseteq T''$, we get that $T'\subseteq T''$ is almost-Pr\"ufer. Therefore, we can define $\psi:[R,S]\to\{(T',T'')\in[R,\overline R]\times[\overline R,S]\mid T'\subseteq T''$ almost-Pr\"ufer$\}$ by $\psi(T):=(T\cap\overline R,\overline RT)$ for each $T\in[R,S]$.

Let $(T',T'')\in[R,\overline R]\times[\overline R,S]$ and assume that $T'\subseteq T''$ is almost-Pr\"ufer. In view of Theorem ~\ref{4.5} applied to the extension $T'\subseteq T''$, there exists $T\in[T',T'']\subseteq[R,S]$ such that $(T',T'')=(T\cap\overline T',\overline T'T)=(T\cap\overline R,\overline RT)$. Hence  $\psi$ is a surjection.

Now, let $T_1,T_2\in[R,S]$ be such that $\psi(T_1)=\psi(T_2)=(T',T'')$ which implies that $T_1,T_2\in[T',T'']$. Another use of Theorem ~\ref{4.5} applied to the extension $T'\subseteq T''$ gives that $T_1=T_2$, so that $\psi$ is injective, and then a bijection.

To end, we have $\psi([R,S])\subseteq[R,\overline R]\times[\overline R,S]$, which implies that $|[R,S]|\leq|[R,\overline R]||[\overline R,S]|$ when $R\subset S$ has FIP. 
\end{proof}

A converse of Theorem~\ref{4.5} is developed  in the following Corollary:

\begin{corollary}\label{4.6} Let $R\subset S$ be an FCP extension. The following statements  are equivalent:
\begin{enumerate}
\item  $R\subset S$ is an almost-Pr\"ufer extension.
\item There exists an order-isomorphism $\psi:[R,S]\to[R,\overline R]\times[\overline R,S]$ defined by $\psi(T):=(T\cap\overline R,\overline RT)$.
 \end{enumerate}

If these conditions hold, then $\mathrm{MSupp}(\overline R/R)=\mathrm{MSupp}(S/\widetilde{R})$ and $\mathrm{MSupp}(\widetilde{R}/R)=\mathrm{MSupp}(S/\overline R)$.

Moreover, if $R\subset S$ is an FIP extension, these conditions are equivalent to:

(3) $|[R,S]|=|[R,\overline R]||[\overline R,S]|$.
\end{corollary}

\begin{proof} (1) $\Rightarrow$ (2) by Proposition ~\ref{split} and Theorem ~\ref{4.5}.

(2) $\Rightarrow$ (1) Since $\psi$ is a bijection, there exists $U\in[R,S]$ such that $\psi(U)=(R,S)=(U\cap\overline R,\overline RU)$. Then $R\subset S$ is almost-Pr\"ufer by Proposition ~\ref{split}. 

If these conditions hold, then $\mathrm{MSupp}(\overline R/R)=\mathrm{MSupp}(S/\widetilde{R})$ and $\mathrm{MSupp}(\widetilde{R}/R)=\mathrm{MSupp}(S/\overline R)$  by Theorem ~\ref{4.5}.

Assume now that $R\subset S$ has FIP.

(2) $\Rightarrow$ (3) obviously.

(3) $\Rightarrow$ (2) Assume that $|[R,S]|=|[R,\overline R]||[\overline R,S]|$ holds. The map $\psi$ which holds in Theorem \ref{4.10} is the same as the map defined in (2). Being injective, it defines a bijection over $[R,\overline R]\times[\overline R,S]$, giving (2).
\end{proof}

We recall the fundamental role of the support in an FCP extension. 

\begin{proposition}\label{2.11} \cite[Proposition 4.1]{DPP3} 
 Let $R\subset S$ be a ring extension.
If there exists a maximal chain $R=R_0\subset\cdots\subset R_i\subset\cdots\subset R_n=S$ of extensions, where $R_i\subset R_{i+1}$ is minimal with crucial ideal $M_i$, (for example, if $R\subseteq S$ is an FCP extension), then $\mathrm{Supp}(S/R)$ is a finite set; in fact, $\mathrm{Supp}(S/R)=\{M_i\cap R\mid i=0,\ldots,n-1\}$.

Moreover, for each $P\in\mathrm{Supp}(S/R)$, we have ${\mathrm V}_R(P)\subseteq \mathrm {Supp}(S/R)$.
\end{proposition}

Let $R\subset S$ be an FCP extension. In \cite[Theorem 4.6]{Pic 5}, we set $\vec R:=\overline R\widetilde{R}= \overline{\widetilde{R}}$, which is the greatest almost-Pr\"ufer subextension of $R\subset S$ (the {\em almost-Pr\"ufer closure} of $R\subset S$). 
In particular, $\overline R$ (resp. $\widetilde{R}$) is also the integral closure (resp. Pr\"ufer hull) of $R$ in $\vec R$.
We get the following Theorem:
 
\begin{theorem}\label{4.09} Let $R\subset S$ be an FCP extension. Then, $\widetilde R=\cap[R_{[M]}\mid M\in\mathrm{MSupp}(\overline R/R)]=R_{[1+(R:\overline R)]}$.
\end{theorem}

\begin{proof} We consider the almost-Pr\"ufer extension $R\subseteq\vec R$. According to Proposition~\ref{split}, $R\subseteq\vec R$ splits at $\overline R$, with $\widetilde{R}=(\overline R)^o$, the complement of  $\overline R$ in $[R,\vec R]$. 

 Set $X:=\mathrm{MSupp}(\overline R/R)$ and $Y:=X^c$, the complement of $X$ in $\mathrm{MSupp}(\vec R/R)$, so that $Y=\mathrm{MSupp}(\widetilde R/R)$ and $\widetilde{R}=\sigma(Y)=\sigma(X^c)$ by Theorem~\ref{4.5}. Then, Proposition \ref{1.154} gives that $\widetilde{R}=\sigma(Y)=\cap[R_{[M,\vec R]}\, \,|\ M\in Y^c]=\cap[R_{[M,\vec R]}\mid M\in\mathrm{MSupp}(\overline R/R)] $, because $X=(Y^c)^c$. In particular, $\widetilde{R}\subseteq \cap[R_{[M]}\mid M\in\mathrm{MSupp}(\overline R/R)]=:R' $.

Assume that $\widetilde{R}\subset R'$, so that there exists $U_1\in[\widetilde{R},R']$ such that $\widetilde {R}\subset U_1$ is a minimal, necessarily integral extension, by maximality of $\widetilde{R}$. Indeed, if $\widetilde{R}\subset U_1$ is  minimal Pr\"ufer, then $R\subset U_1$ is Pr\"ufer, a contradiction. It follows that there exists $R_1\in[R,U_1]$ such that $R\subset R_1$ is a minimal integral extension. Let $x\in R_1$ be such that $R_1=R[x]$ and set $M:=\mathcal{C}(R,R_1)\in\mathrm{MSupp}(R_1/R)\subseteq\mathrm{MSupp}(\overline{R}/R)$. Since $x\in R'$, it follows that there exists $s\in R\setminus M$ such that $xs\in R$, giving $x/1\in R_M$, so that $(R_1)_M=R_M$, a contradiction. Then, $\widetilde{R}=R'$.

Set $I:=(R:\overline{R})$. Since $R\subseteq\overline{R}$ has FCP, \cite[Theorem 4.2]{DPP2} implies that $R/I$ is an Artinian ring; so that,  ${\mathrm V}(I):=\{M_1,\ldots,M_n\}$ is a finite set. By the previous equality, we get that $\widetilde{R}=\cap_{M_i\in{\mathrm V}(I)}R_{[M_i]}$ since ${\mathrm V}(I)=\mathrm{MSupp}(\overline{R}/R)$ (\cite[Proposition 4.1]{DPP3}). 

Let $y\in\widetilde{R}$. For each $i\in\mathbb N_n$, we have $y\in R_{[M_i]}$, and there exists $\lambda_i\in R\setminus M_i$ such that $\lambda_iy\in R$. But, in $R/I$, we have $\mathrm{Spec}(R/I)=\{M_1/I,\ldots,M_n/I\}={\mathrm D}(\overline{\lambda}_1)\cup\cdots\cup{\mathrm D}(\overline{\lambda}_n)$, where $\overline{\lambda}_i$ is the class of $\lambda _i$ in $R/I$. Hence, $R/I=\sum_{i=1}^n(R/I)\overline {\lambda}_i$, and there exist $\mu_1,\ldots,\mu_n\in R$ such that $x:=1-\sum_{i=1}^n{\lambda}_i\mu_i\in I$, giving that $y-\sum_{i=1}^ny{\lambda}_i\mu_i=xy$. But, $y{\lambda}_i\in R$ for each $i$ implies that $(1-x)y=\sum_{i=1}^ny{\lambda}_i\mu _i\in R$, with $1-x\in 1+I$; so that $y\in R_{[1+I]}$. 

Conversely, if $y\in R_{[1+I]}$, there exists some $x\in I$ such that $(1+x)y\in R$, with $x\in M_i$ for each $i$, whence $1+x\not\in M_i$ for each $i$. Then, for each $i$, we get that $y\in R_{[M_i]}$ and $y\in\cap_{M\in\mathrm{MSupp}(\overline{R}/R)}R_{[M]}=\widetilde{R}$. To conclude, we have $\widetilde{R}=R_{[1+I]}$.  
\end{proof}

\begin{corollary}\label{0.148} Let $R\subset S$ be an FCP extension and $T\in]R,\overline R]$. Then, $\widetilde R=R_{[1+(R:T)]}$.\end{corollary}

\begin{proof} Use Theorem \ref{4.09} with $S':=\widetilde RT$. Then $\widetilde R$ is also the Pr\"ufer hull of $R\subset S'$ and $T$ is the integral closure of $R\subset S'$.
\end{proof}

\begin{corollary}\label{prufclos} If  $R\subset S$ is  an FCP extension such that $(R:\overline{R})=0$, then, $\widetilde {R}=R$; so that $R$ is Pr\"ufer closed in $S$.
\end{corollary}

\begin{proof} Obvious, since $\widetilde{R}=R_{[1]}=R$.
\end{proof}

 \begin{corollary} \label{0.149}
 An FCP extension $R\subseteq S$, with conductor $I:=(R:\overline R)$ is almost-Pr\"ufer if and only if $R_{[1+I]}\subseteq S$ is integral, in which case $\overline R \subseteq S=\overline{R}_{[1+I]}$ is Pr\"ufer.
\end{corollary}
\begin{proof}
The first part follows from Theorem \ref{4.09}. Set $\Sigma:=1+I$ and assume that the extension is almost-Pr\"ufer. Let $x\in S\setminus \overline R$. There exist a positive integer $n>1$ and $a_0,\ldots,a_{n-1}\in R_{[\Sigma]}$ such that $x^n+\sum_{i=0}^{n-1}a_ix^{i}=0\ (*)$. For each $i\in\{0,\ldots,n-1\}$, there exists $s_i\in\Sigma$ such that $b_i:=a_is_i\in R$. Set $s:=\prod_{i=0}^{n-1}s_i=s_is'_i$, where $s'_i\in\Sigma$ for each $i\in\{0,\ldots,n-1\}$. Multiplying the two sides of $(*)$ by $s^n$, we get $(sx)^n+\sum_{i=0}^{n-1}s^{n-1-i}s'_ib_i(sx)^{i}=0\ (**)$, so that $sx\in\overline{R}$. To sum up, we have proved that for any $x\in S$, there exists some $s\in \Sigma$ such that $sx\in\overline{R}$, which means that $S=\overline{R}_{[\Sigma]}$. Since an almost-Pr\"ufer extension is quasi-Pr\"ufer, $\overline R \subseteq S=\overline{R}_{[\Sigma]}$ is Pr\"ufer.
\end{proof}

The following lemma gives a new result on  $\vec R$.

\begin{lemma}\label{4.12} Let $R\subset S$ be an FCP extension such that $\overline R\neq S$. Then, $\vec R\neq\overline R$ if and only if $\mathrm{MSupp}_R(\overline R /R)\neq\mathrm {MSupp}_R(S/R)$,  in which case $[\overline R,\vec R]=\{T'\in[\overline R,S]\mid R\subseteq T'$ almost-Pr\"ufer$\}$.
\end{lemma}

\begin{proof} Obviously, $\mathrm{MSupp}_R(S/R)=\mathrm{MSupp}_R(\overline R/R)\cup\mathrm {MSupp}_R(S/\overline R)$. Set $T:=\vec R$. Assume first that $T\neq\overline R$.
Since $\overline R$ is also the integral closure of $R$ in $T$, with $R\subset T$ almost-Pr\"ufer, we have $\mathrm{MSupp}(\overline R/R)\cap\mathrm{MSupp}(T/\overline R)=\emptyset$. Moreover, there exists $R'\in[\overline R,T]$ such that $\overline R\subset R'$ is a minimal Pr\"ufer extension. Set $N:=\mathcal{C}(\overline R,R')\in\mathrm{Max}(\overline R)$ and $M:=N\cap R$. Then $M\in\mathrm{MSupp}(T/\overline R)$ in view of Proposition~\ref{2.11}. The assumption gives that $M\not\in\mathrm{MSupp}(\overline R/R)$, so that $\mathrm{MSupp}(\overline R/R)\neq\mathrm {MSupp}(S/R)$. 

Conversely, assume that $\mathrm{MSupp}(\overline R/R)\neq\mathrm{MSupp}(S/R)$ and let $M\in\mathrm{MSupp}(S/R)\setminus\mathrm{MSupp}(\overline R/R)$, so that $M\in\mathrm{MSupp}(S/\overline R)$. Let $\mathcal C$ be a maximal chain defined by $\overline R=R_0\subset\cdots\subset R_i\subset\cdots\subset R_n=S$, where $R_i\subset R_{i+1}$ is minimal with $M_i:=\mathcal{C}(R_i,R_{i+ 1})$. There exists some $k$  such that  $M_k\cap R=M$.  
 Set $N:=M_k\cap \overline R\in\mathrm{Spec}(\overline R)$. Then,  $M=N\cap R$ with $N\in\mathrm{Max}(\overline R)$ because $R\subseteq\overline R$ is integral and $N\in\mathrm{MSupp}_{\overline R}(S/\overline R)$. According to \cite[Lemma 1.8]{Pic 6}, there exists $T_1\in[\overline R,S]$ such that $\overline R\subset T_1$ is Pr\"ufer minimal with $N=\mathcal{C}(\overline R,T_1)$, so that $\{N\}=\mathrm {MSupp}_{\overline R}(T_1/\overline R)$, leading to $\{M\}=\mathrm {MSupp}(T_1/\overline R)$; 
   so that, $\mathrm{MSupp}(T_1/\overline R)\cap\mathrm{MSupp}(\overline R/R)=\emptyset$ because $M\not\in\mathrm{MSupp}(\overline R/R)$. It follows that $R\subset T_1$ is almost-Pr\"ufer and $T_1\in[\overline R,T]$.
  To conclude $\vec R\neq\overline R$. 

The last assertion comes from $\mathrm{MSupp}_R(T'/\overline R)\subseteq\mathrm{MSupp}_R(T/\overline R)$ for any $T'\in[\overline R,T]$.
\end{proof}

The above lemma allows us  to calculate the cardinality of $[R,S]$ for an FIP extension $R\subset S$.

\begin{corollary}\label{4.13} Let $R\subset S$ be an FIP extension with $\mathrm{MSupp}_R(\overline R/R)\neq\mathrm{MSupp}_R(S/R)$. Set $[R,\overline R]:=\{R_i\}_{i=1,\ldots,n}$ and $[\overline R,S]:=\{R'_j\}_{j=1,\ldots,p}$. Then $|[R,S]|=\sum_{j=1}^pr_j=\sum_{i=1}^ns_i$, where

\noindent $r_j:=|\{R_i\in[R,\overline R]\mid R_i\subseteq R'_j$  almost-Pr\"ufer$\}|$ for each $j$, and

\noindent $s_i:=|\{R'_j\in[\overline R,S]\mid R_i\subseteq R'_j$  almost-Pr\"ufer$\}|$ for each $i$.

More precisely, $s_i=|[\overline R,T_i]|$, where $T_i$ is the almost-Pr\"ufer closure of $R_i\subset S$.
\end{corollary}

\begin{proof} In Theorem~\ref{4.10}, we define the bijection $\psi:[R,S]\to 
\{(T',T'')\in[R,\overline R]\times[\overline R,S]\mid T'\subseteq T''$ almost-Pr\"ufer$\}$ by $\psi(T):=(T\cap\overline R,\overline RT)$, so that $|[R,S]|=|\psi([R,S])|$. But $\psi([R,S])=\cup_{j=1}^p\{(R_i,R'_j)\in[R,\overline R]\times[\overline R,S]\mid R_i\subseteq R'_j$ almost-Pr\"ufer$\}=\cup_{i=1}^n\{(R_i,R'_j)\in[R,\overline R]\times[\overline R,S]\mid R_i\subseteq R'_j$ almost-Pr\"ufer$\}$. This gives the first result.

Fix some $R_i\in[R,\overline R],\ R_i\neq\overline R$. The assumption gives $\overline R\neq S$. Then, $\overline R$ is also the integral closure of $R_i$ in $S$. We may apply Lemma~\ref{4.12} to the extension $R_i\subseteq S$, because $\mathrm{MSupp}_{R_i}(\overline R/R_i)\subseteq\mathrm{MSupp}_{R_ i}(S/R_i)$. If $\mathrm{MSupp}_{R_i}(\overline R/R_i)=\mathrm{MSupp}_{R_i}(S/R _i)$, then, $N\in\mathrm{MSupp}_{R_i}(\overline R/R_i)$ for any $N\in\mathrm{MSupp}_{R_i}(S/R_i)$. Let $R_i=:T_0\subset\cdots\subset T_l:=\overline R\subset\cdots\subset T_ m:=S$ be a maximal chain of extensions such that $T_k\subset T_{k+1}$ is a minimal extension for each $k\in\{0,\ldots,m-1\}$. Set $M_k:=\mathcal C(T_k,T_{k+1})$. This means that for each $k\geq l$, there exists $k'<l$ such that $ M_k\cap R_i=M_{k'}\cap R_i$, so that $M_k\cap R=M_{k'}\cap R\in\mathrm {MSupp}_R(\overline R/R)$. Since $\mathrm{MSupp}_R(\overline R/R)\neq\mathrm{MSupp}_R(S/R)$, there exists some $k\geq l$ such that $M_k\cap R\not\in\mathrm{MSupp}_R(\overline R/R)$. In particular, $M_k\cap R_i\not\in\mathrm{MSupp}_{R_i}(\overline R/R_i)$ for any $i<k$ while $M_k\cap R_i\in\mathrm{MSupp}_{R_i}(S/R_i)$, a contradiction. It follows that $\mathrm{MSupp}_{R_i}(\overline R/R_i)\neq \mathrm{MSupp}_{R_i}(S/R _i)$.

Let $T_i$ be the almost-Pr\"ufer closure of $R_i\subset S$. By Lemma~\ref{4.12}, we have $[\overline R,T_i]=\{R'_j\in[\overline R,S]\mid R_i\subseteq R'_j$ almost-Pr\"ufer$\}$, so that $s_i=|[\overline R,T_i]|$.

We may remark that, when $R_i=\overline R$, we get $s_i=|[\overline R,S]|$, since $\mathrm{MSupp}_{\overline R}(\overline R/\overline R)=\emptyset$.
\end{proof}

 Let $R\subset S$ be an FCP extension. Since $R\subseteq \vec R$ is almost-Pr\"ufer, we may apply the previous results to the extension $R\subseteq\vec R$, keeping in mind that $\overline{R}$ (resp. $\widetilde {R}$) is also the integral closure (resp. Pr\"ufer hull) of $R\subseteq \vec R$.

\begin{proposition}\label{7.2} Let $R\subset S$ be a FCP extension. The following statements hold:
\begin{enumerate}
\item The maps $\psi:[R,\vec R]\to[R,\overline{R}]\times[\overline{R},\vec R]$ defined by $\psi(T):=(T\cap\overline{R},T\overline{R})$ and $\psi':[R,\vec R]\to[R,\widetilde{R}]\times[\widetilde{R},\vec R]$ defined by $\psi'(T):=(T\cap\widetilde{R},T\widetilde {R})$ for any $T\in[R,\vec R]$ are order-isomorphisms.

\item $\mathrm{MSupp}(\vec R/\widetilde{R})=\mathrm{MSupp}(\overline{R}/R)$, $\mathrm{MSupp}(\vec R/\overline{R})=\mathrm{MSupp}(\widetilde{R}/R)$ and
 $\mathrm{MSupp}(\vec R/R)=\mathrm{MSupp}(\widetilde{R}/R)\cup \mathrm{MSupp}(\overline{R}/R)$.

\item The maps $\psi_1:[R,\widetilde{R}]\to[\overline{R},\vec R]$ defined by $\psi_1(T):=T\overline {R}$ and $\psi'_1:[R,\overline{R}]\to[\widetilde{R},\vec R]$ defined by $\psi'_1(T):=T\widetilde {R}$ are order-isomorphisms.

\item The map $\theta:[R,\widetilde{R}]\times[R,\overline{R}]\to[R,\
\vec R]$ defined by $\theta(T,T'):=TT'$, is an order-isomorphism. In particular, if $R\subset S$ has FIP, then $ |[R,\vec R] \|=|[R,\widetilde{R}]||[R,\overline{R}]|$.
\end{enumerate}
\end{proposition}

\begin{proof} (1) and (2) follow from Proposition ~\ref{split}, Theorem ~\ref{4.5}  and Corollary ~\ref{4.6}. 

(3) We begin to remark that $\overline{R}$ and $\widetilde{R}$ play symmetric roles. Then, use Proposition ~\ref{1.169} (1).

(4) is Proposition ~\ref{1.169} (2).  The FIP case is obvious.
\end{proof}

\begin{corollary}\label{simplifiable 1}Let $R\subset S$ be an FCP ring extension. The following properties hold:
\begin{enumerate}
\item $U\overline{R}=V\overline{R}$ implies $U=V$ for any $U,V\in[R,\widetilde{R}]$.
\item $U\cap \overline{R}=V\cap \overline{R}$ implies $U=V$ for any $U,V\in[\widetilde{R},\vec R]$.
\item $U\widetilde{R}=V\widetilde{R}$ implies $U=V$ for any $U,V\in[R,\overline{R}]$.
\item $U\cap \widetilde{R}=V\cap \widetilde{R}$ implies $U=V$ for any $U,V\in[\overline{R},\vec R]$.
\end{enumerate}
\end{corollary}
\begin{proof} Apply Corollary \ref{simplifiable} to the extension $R\subseteq\vec R$ since $R\subseteq\vec R$ splits at $\overline{R}$ and $\widetilde{R}$.
\end{proof}

\begin{corollary}\label{simplifiable 2}Let $R\subset S$ be an FCP ring extension and let $T,U\in[R,S]$ be such that $R\subset T$ is integral and $R\subset U$ is Pr\"ufer.    
\begin{enumerate}
\item $VT=WT$ implies $V=W$ for any $V,W\in[R,U]$.
\item $V\cap U=W\cap U$ implies $V=W$ for any $V,W\in[R,T]$.
\end{enumerate}
\end{corollary}
\begin{proof} Use Corollary \ref{simplifiable 1} applied to the extension $R\subset TU$ because $T$ (resp. $U$) is the integral closure (resp. Pr\"ufer hull) of the almost-Pr\"ufer extension $R\subset TU$.   
\end{proof}

Gathering the previous results, we get new equivalences for an FCP extension to be almost-Pr\"ufer.

\begin{theorem}\label{7.6} Let $R\subset S$ be an FCP extension. The following conditions are equivalent:
\begin{enumerate}
\item  $R\subseteq S$ is almost-Pr\"ufer;

\item $\mathrm{MSupp}(S/\overline{R})=\mathrm{MSupp}(\widetilde{R}/R)$;

\item The map $\psi_1:[R,\widetilde{R}]\to[\overline{R},S]$ defined by $\psi_1(T):=T\overline {R}$ is an order-isomorphism;

\item The map  $\psi'_1:[R,\overline{R}]\to[\widetilde{R},S]$ defined by $\psi'_1(T):=T\widetilde {R}$ is an order-isomorphism;

\item The map  $\theta:[R,\widetilde{R}]\times[R,\overline{R}]\to[R,S]$ defined by $\theta(T,T'):=TT'$ is an order-isomorphism.
\end{enumerate}

\noindent If one of these conditions holds, then $\mathrm{MSupp}(S/\widetilde{R})=\mathrm{MSupp}(\overline{R}/R)$.

If the extension has FIP, the preceding  conditions are equivalent to each of the following conditions:
\begin{enumerate}
\item[(6)] $|[R,S]|=|[R,\widetilde{R}]||[R,\overline{R}]|$;

\item[(7)] $|[R,\widetilde{R}]|=|[\overline{R},S]|$;

\item[(8)] $|[R,\overline{R}]|=|[\widetilde{R},S]|$.
\end{enumerate}
\end{theorem}

\begin{proof} First, $R\subset S$ is an almost-Pr\"ufer if and only if $S=\vec R$ by definition of $\vec R$.

(1) $\Rightarrow$ (2), (3), (4) and (5): Use Proposition~\ref{7.2} (2) to get (2), Proposition~\ref{7.2} (3) to get (3) and (4), and Proposition~\ref{7.2} (4) to get (5). Moreover, Proposition~\ref{7.2} (2) give $\mathrm{MSupp}(S/\widetilde{R})=\mathrm{MSupp}(\overline{R}/R)$. 

(2) $\Rightarrow$ (1): Since $\mathrm{MSupp}(\widetilde{R}/R)\cap \mathrm{MSupp}(\overline{R}/R)=\emptyset$, by \cite[Proposition 4.18]{Pic 5}, (2) leads to $\mathrm{MSupp}(S/\overline{R})\cap\mathrm{MSupp}(\overline{R}/R)=\emptyset$, so that $R\subset S$ splits at $\overline{R}$. Then, use  Proposition~\ref{split}. 

(3), (4) or (5) $\Rightarrow$ (1) because, in each case, we have $S=\overline{R}\widetilde{R}=\vec R$.

Assume now that $R\subset S$ has FIP. 

Then, obviously, (5) $\Rightarrow$ (6), (3) $\Rightarrow$ (7) and (4) $\Rightarrow$ (8). 

(6) $\Rightarrow$ (1): Assume that $|[R,S]|=|[R,\widetilde{R}]||[R,\overline{R}]|$. By Proposition~\ref{7.2} (4), it follows that $|[R,S]|=|[R,\vec R]|$, so that $S=\vec R$.

(7) $\Rightarrow$ (1) and (8) $\Rightarrow$ (1) by Proposition~\ref{7.2} (3) because in each case, we get $S=\vec R$ .
\end{proof}

\begin{example} We give an example where the equivalences  of Theorem~\ref{7.6} do not hold because $R\subseteq S$ has not FCP. Set $R:=\mathbb{Z}_P$ and $S:=\mathbb{Q}[X]/(X^2)$, where $P\in \mathrm{Max}(\mathbb{Z})$. Then, $\widetilde R=\mathbb{Q}$ because $R\subset\widetilde R$ is Pr\"ufer (minimal) and $\widetilde R\subset S$ is integral minimal, so that $R\subset S$ is almost-Pr\"ufer. Set $M:=PR\in\mathrm{Max}(R)$ with $(R,M)$ a local ring. It follows that $M\in\mathrm{MSupp}(\overline{R}/R)\cap\mathrm{MSupp}(S/\overline{R})$. Similarly, $M\in\mathrm{MSupp}(\overline{R}/R)\cap\mathrm{MSupp}(\widetilde R/R)$. Indeed, $R\subset \overline R$ has not FCP. 
\end{example}

By using the  results of the paper, we build the almost-Pr\"ufer closure $\vec R$ of an FIP extension $R\subset S$ in the context of algebraic orders. In this example, we determine the Pr\"ufer hull and the integral closure of the extension. These three rings are  distinct and distinct from $R$ and $S$.
 
\begin{example}\label{0.23} Let $K$ be an algebraic number field, $T'$ be its ring of algebraic integers, and $R'$ be an algebraic order of $K$ such that $R'\subset T'$ is a minimal inert extension. Set $M':=(R':T')\in\mathrm{Max}(R')$, so that $M'\in\mathrm{Max}(T')$. Let $N'\in\mathrm{Max}(R'),\ M'\neq N'$. Set $\Sigma:=R'\setminus(M'\cup N')$, which is a multiplicatively closed subset. Then, $R:=R'_{\Sigma}$ is a semilocal ring with two maximal ideals $M:=M'_{\Sigma}$ and $N:=N'_{\Sigma}$. Moreover $T:=T'_{\Sigma}$ is also a semilocal ring with two maximal ideals, and such that $R\subset T$ is a minimal inert extension with $M=M'_{\Sigma}=(R:T)\in\mathrm{Max}(T)$. It verifies $\{M\}=\mathrm{MSupp}(T/R)$. The other maximal ideal of $T$ is $P:=NT$ thanks to \cite[Lemma 2.4]{DPP2}. Moreover, $T$ is a Pr\"ufer domain with quotient field $K$. According to \cite[Corollary 2.5]{D2}, it follows that there exists a maximal chain $T\subset T_1\subset K$ of length 2 and $T\subset K$ has FCP (and FIP) by \cite[Theorem 6.3]{DPP2}. In fact, $\ell[T,K]=2$ and any maximal chain of $T\subset K$ has length 2 by \cite[Corollary 6.13]{DPP2}. As $T_M\in]T,K[$, it follows that $T\subset T_M$ is minimal Pr\"ufer with $P=\mathcal{C}(T,T_M)$ since $PT_M=T_M$. In fact, $\{N\}=\mathrm{MSupp}_R(T_M/T)$. Therefore $T$ is the integral closure of $R\subset T_M$, and \cite[Theorem 3.13]{DPP2} gives that $R\subset T_M$ is an FCP extension. For the same reason, $T$ is the integral closure $\overline R$ of $R\subset K$, and $R\subset K$ is an FCP extension. More precisely, $R\subset K$ is an FIP extension because $R\subset T$ and $T\subset K$ have FIP (\cite[Theorem 3.13]{DPP2}). We have the following commutative diagram:

\centerline{$\begin{matrix}
 {} &       {}       & R_M &      {}       &   {}   &      {}      & {} \\
 {} & \nearrow &    {}   & \searrow &   {}   &      {}       & {} \\
 R &      {}        &   {}   &       {}      & T_M &      \to      & K \\
 {} & \searrow &   {}   & \nearrow &   {}    & \nearrow & {} \\
 {} &      {}       &   T   &       \to      & T_P  &       {}      & {}
\end{matrix}$}
Because $\mathrm{MSupp}(T/R)\cap\mathrm{MSupp}(T_M/T)=\{M\}\cap\{N\}=\emptyset$, we get that $R\subset T_M$ is almost-Pr\"ufer by Proposition~\ref{split}. Applying Theorem \ref{7.6} to the almost-Pr\"ufer extension $R\subset T_M$, we deduce that $\widetilde{R}\subset T_M$ is minimal integral. But $R_M\subset T_M$ is integral because so is $R\subset T$ with $M=(R:T)$. Because $\widetilde{R}\subseteq R_M\subset T_M$ is minimal, we get that $\widetilde{R}=R_M$, so that $R_M$ is the Pr\"ufer hull of $R\subset T_M$. In the same way, Theorem \ref{7.6} shows that $R\subset R_M$ is Pr\"ufer minimal. 
But $R_M$ is also the Pr\"ufer hull of $R\subset K$ since there does not exist some $W\in]R_M,K[$ such that $R_M\subset W$ is Pr\"ufer. Otherwise, $W$ would be a zero-dimensional integral domain with quotient field $K$, a contradiction. 
Then, $R\subset K$ is not almost-Pr\"ufer, since $R_M\subset K$ is not integral. And $T_M=\vec R$ is the almost-Pr\"ufer closure of $R\subset K$. 

We may apply Corollary \ref{4.13} to the extension $R\subset K$. We already knows that $R\subset K$ has FIP. Moreover, $[R,\overline R]=\{R,T\}$ and $[\overline R,K]=\{T,T_M,T_P,K\}$, so that, setting $R_1:=R$ and $R_2:=T$, we get $|[R,K]|=s_1+s_2$, with $T_1:=T_M$, the almost-Pr\"ufer closure of $R\subset K$ and $T_2:=K$, the almost-Pr\"ufer closure of $T\subset K$. It follows that $s_1=|[T,T_M]|=2$ and $s_2=|[T,K]|=4$ by Proposition \ref{6.12} since $\mathrm{Supp}_T(K/T)=\{M,P\}\subseteq\mathrm{Max}(T)$, so that $|[R,K]|=6$ and $[R,K]=\{R,R_M,T,T_M,T_N,K\}$. We may remark that $R_N=T_N=T_P$.
 Moreover, $R\subset K$ is a $\mathcal B$-extension since $\mathrm{Supp}(K/R)=\{M,N\}=\mathrm{Max}(R)$, but $R$ is not a pm-ring.

\end{example} 

We end by   some length computations in the FCP context.

\begin{proposition}\label{7.8} Let $R\subseteq S$ be an FCP extension. The following statements hold:
\begin{enumerate}
\item $\ell[R,\widetilde R]=\ell[\overline R,\vec R]$ and $\ell[R,\overline R] = \ell [\widetilde R,\vec R]$

\item $\ell[R,\vec R]=\ell[R,\widetilde R]+\ell[\widetilde R,\vec R]=\ell[R,\overline R] + \ell [\overline R, \vec R]$

\item  $\ell[\overline R,\vec R]=|\mathrm{Supp}_{\overline R}(\vec R/ \overline R)|=\ell [R, \widetilde R] = |\mathrm{Supp}_R (\widetilde R/R)|$.
\end{enumerate}
\end{proposition}

\begin{proof} To prove (1), use the maps $\psi_1$ and $\psi'_1$ of Proposition~\ref{7.2} (3). Then (2) follows from \cite[Theorem 4.11]{DPP3} and (3) from \cite[Proposition 6.12]{DPP2}.
\end{proof}   

\section{Pinched extensions versus split extensions}

 We recall that a ring extension $R\subset S$ is called {\em pinched} at $T\in]R,S[$ if $[R,S]=[R,T]\cup[T,S]$. 

\begin{proposition}\label{1.145}Let $R\subset S$ be an FCP  ring extension. Then, any splitter is trivial in the following cases:
\begin{enumerate}
\item $R\subset S$ is pinched at some $U\in]R,S[$.

\item $R\subset S$ is chained.
\end{enumerate}
\end{proposition}
\begin{proof}(1) Assume that there exists a proper splitter $T$ of $R\subset S$. By Theorem \ref{4.5}, $T$ has a complement $T^o$ in $[R,S]$ such that $T\cap T^o=R$ and $TT^o=S\ (*)$. If $T,T^o\in [R,U]$, then $S=TT^o\subseteq U\subset S$, a contradiction. If $T,T^o\in [U,S]$, then $R\subset U\subseteq T\cap T^o=R$, a contradiction. Now, if $T$ and $T^o$ are not both in either $[R,U]$ or $[U,S]$, one of then is contained in $U$ while the other contains $U$, so that $T^o$ is  comparable to $T$, a contradiction with $T$ is a proper splitter by $(*)$.

(2) A chained extension is pinched at any of its subextension. Then, apply (1) if $R\subset S$ is not minimal. If $R\subset S$ is  minimal, the result is obvious.
\end{proof}

\begin{corollary}\label{1.162} An FCP ring extension  $R\subset S$  is not pinched at any proper splitter.
\end{corollary} 
\begin{proof} 
See Proposition \ref{1.145}.
\end{proof}

We have seen in Proposition ~\ref{split}, that when an FCP extension $R\subset S$ is almost-Pr\"ufer and $\overline R\neq R,S$, then $\overline R$ is a non trivial splitter. This implies by Proposition ~\ref{1.162} that $R\subset S$ is not pinched at $\overline R$. We end this paper by a new characterization of FCP extension $R\subset S$ pinched at $\overline R$, improving our earlier result \cite[Proposition 2.7]{Pic 13}, because it avoids to consider minimal extensions of the form $U\subset \overline R$, by only using the properties of the canonical decomposition (Definition \ref{canonical}). 
  We begin by recaling \cite[Proposition 2.7]{Pic 13} and  then give  needed   Lemmata.
 
\begin{proposition}\label{1.131} An FCP extension $R\subseteq S$,  such that $\overline R\neq R,S$,  is pinched at $\overline R$ if and only if, for any $U\in[R,S]$ such that $U\subset\overline R$ is minimal, then $\mathrm{MSupp}_{\overline R}(S/\overline R)\subseteq \mathrm{V}_{\overline R}((U:\overline R))$.  \end{proposition}

\begin{lemma}\label{1.132} Let $R\subset S$ be an $M$-crucial integral FCP extension with conductor $I$ and $N\in \mathrm{V}_S(I)$. There is some $U\in[R,S]$ such that $U\subset S$ is minimal decomposed with $N\in\mathrm{V}_S((U:S))$ if and only if there exists $N'\in \mathrm{V}_S(I),\ N'\neq N$ such that $S/N'\cong S/N$ as $R/M$-algebras.
\end{lemma}
\begin{proof} First, since $R\subset S$ is $M$-crucial, $(U:S)\cap R\subseteq M$, so that $N$ lies above $M$.

Assume  that there is some $U\in[R,S]$ such that $U\subset S$ is minimal decomposed with $N\in\mathrm{V}_S((U:S))$. Then, $P:=(U:S)\in \mathrm{Max}(U)$. According to Theorem \ref{minimal}, there exists  
 $N'\in\mathrm{Max}(S),\ N'\neq N$ such that $S/N'\cong S/N$ as $U/P$-algebras. Since $M\subseteq P$, it follows that $S/N'\cong S/N$ as $R/M$-algebras.
 
 Conversely, if there exists $N'\in \mathrm{V}_S(I),\ N'\neq N$ such that $S/N'\cong S/N$ as $R/M$-algebras, \cite[Lemma 20]{DPP4} shows that  that there is some $U\in[R,S]$ such that $U\subset S$ is minimal decomposed with $(U:S)=N\cap N'$, so that $N\in\mathrm{V}_S((U:S))$.
\end{proof}   

\begin{lemma}\label{7.9} 
Let $R\subset S$
 be an $M$-crucial 
 integral FCP extension with conductor $I$ and $N\in \mathrm{V}_S(I)$. There is some $U\in[R,S]$ such that $U\subset S$ is minimal ramified with $\mathrm{V}_S((U:S))=\{N\}$  if and only if 
 $N\in\mathrm{MSupp}_S(N/MS)$.
\end{lemma}
\begin{proof}  
Since $MI\subseteq NI\subseteq I$, we may assume that $I=0$. Then, according to \cite[Theorem 4.2]{DPP2}, $R$ is an Artinian ring, as $S$ and $R\subseteq S$ is an $M$-crucial FCP integral extension. To prove the statement of the Lemma, it is enough to show that there exists some $U\in[R,S]$ such that $U\subset S$ is minimal ramified with $\mathrm{V}_S((U:S))=\{N\}$ if and only if $N\in\mathrm{MSupp}_S(N/MS)$. We use the proof of \cite[Lemma 17]{DPP4}. 

We begin to show that $N\not\in\mathrm{MSupp}_S(N/MS)$ if and only if the $N$-primary component of $MS$ in $S$ is $N$. This is equivalent to show that $N\in\mathrm{MSupp}_S(N/MS)$ if and only if the $N$-primary component of $MS$ in $S$ is different from $N$. Since $S$ is an Artinian ring, let $MS=(\cap_{i\in\Lambda}Q_i)\cap Q\ (*)$ be the primary decomposition of $MS$ into primary ideals of $S$, with $Q$ its $N$-primary component, so that $N$ is the only maximal ideal of $S$ containing $Q$. Localizing $(*)$ at $N$, it follows that $(MS)_N= Q_N$. Then, $Q=N\Leftrightarrow(MS)_N=N_N\Leftrightarrow (N/MS)_N=0\Leftrightarrow N\not\in\mathrm{MSupp}_S(N/MS)$.

Assume first that there exists some $U\in[R,S]$ such that $U\subset S$ is minimal ramified. Then, $(U:S)\in\mathrm{Max}(U)$ and is contained in a unique maximal ideal $N$ of $S$. Set  $P:=(U:S)$, which is an $N$-primary ideal of $S$ and satisfies $M\subseteq P\subset N$ because $R\subseteq S$ is $M$-crucial. It follows that $MS\subseteq P$, because $P$ is also an ideal of $S$. Assume that $MS=(\cap_{i\in\Lambda} P_i)\cap N$ is the primary decomposition of $MS$ in $S$, where $P_i$ is $N_i$-primary for each $i\in\Lambda$. Hence, $(\cap_{i\in\Lambda}P_i)\cap N\subseteq P\ (**)$, a contradiction with the uniqueness of the primary decomposition (it is enough to localize $(**)$ at $N$). Then, the $N$-primary component of $MS$ is different from $N$.

Conversely, assume that the $N$-primary component of $MS$ is different from $N$. Let $MS=(\cap_{i\in\Lambda}Q_i)\cap Q$ be the primary decomposition of $MS$ in $S$,
 where $Q_i$ is $N_i$-primary for any $i\in\Lambda$ and $Q$ is $N$-primary, with $Q\subset N$. Then, there exists an $N$-primary ideal $P$ of $S$ such that $Q\subseteq P\subset N$, where $P$ and $N$ are adjacent ideals. We use again the proof of \cite[Lemma 17]{DPP4}. Then, $S/P$ has a field of representatives $K\cong S/N$ which contains $R/M$. Let $\psi:S\to S/P$ be the canonical surjection and set $T:=\psi^{-1}(K)$. It follows that $T\subset S$ is a minimal ramified extension with $(T:S)=P$, an $N$-primary ideal of $S$. This completes the proof. 
\end{proof}

\begin{proposition}\label{7.10} An FCP extension $R\subseteq S$, such that $\overline R\neq R,S$, is pinched at $\overline R$ if and only if 
$R\subseteq S$ is an $M$-crucial extension, for some $M\in\mathrm{Max}(R)$, $\mathrm{MSupp}_{\overline R}(S/\overline R)\subseteq\mathrm{V}_{\overline R}((R:\overline R))$ and one of the following conditions holds:

\begin{enumerate}
\item $|\mathrm{MSupp}_{\overline R}(S/\overline R)|=2,\ R\subseteq\overline R$ is infra-integral and $M\overline R=N_1\cap N_2$, where $\mathrm{MSupp}_{\overline R}(S/\overline R)=\{N_1,N_2\}$;

\item ${}_{\overline R}^tR\subset\overline R$ and $\overline R\subseteq S$ are $N$-crucial extensions, where
 $\{N\}=\mathrm{MSupp}_{\overline R}(S/\overline R)=\mathrm{MSupp}_{{}_{\overline R}^tR}(\overline R/{}_{\overline R}^tR)$, $\overline R/N'\not\cong\overline R/N''$ for any distinct $N',N''\in\mathrm{V}_{\overline R}((R:\overline R))$, except for at most one pair $(N,N'),\ N'\neq N$ and, 
for any $N''\in\mathrm{V}_{\overline R}((R:\overline R))$ such that $ N''\neq N$, then  $N''\not\in\mathrm{MSupp}_{\overline R}(N''/M\overline R)$;

\item $\overline R\subseteq S$ is an $N$-crucial extension, where $\mathrm{MSupp}_{\overline R}(S/\overline R)=\{N\}$, 
$|\mathrm{V}_{\overline R}((R:\overline R))|=2$, with 
$N\in\mathrm{V}_{\overline R}((R:\overline R))$.
 Setting $\mathrm{V}_{\overline R}((R:\overline R))=\{N,N'\},\ R\subseteq\overline R$ is infra-integral and $N'\not\in\mathrm{MSupp}_{\overline R}(N'/M{\overline R})$;

\item $\overline R\subseteq S$ is crucial, and $R\subseteq\overline R$ is subintegral.
\end{enumerate}  
\end{proposition}

\begin{proof} We use Proposition \ref{1.131}, which says that $R\subset S$ is pinched at $\overline R$ if and only if for any $U\in[R,S]$ such that $U\subset\overline R$ is minimal, then $\mathrm{MSupp}_{\overline R}(S/\overline R)\subseteq\mathrm{V}_{\overline R}((U:\overline R))$. In the two parts of the proof, we discuss according to the type of the minimal extension  $U\subset\overline R$.

Assume first that $R\subset S$ is pinched at $\overline R$ and let $U\in[R,S]$ be such that $U\subset\overline R$ is minimal. Then $\mathrm{MSupp}_{\overline R}(S/\overline R)\subseteq\mathrm{V}_{\overline R}((U:\overline R))$. For any $N\in\mathrm{MSupp}_{\overline R}(S/\overline R)$, we have $(U:\overline R)\subseteq N$.  Since  $(R:\overline R)\subseteq(U:\overline R)\subseteq N$, then $\mathrm{MSupp}_{\overline R}(S/\overline R)\subseteq\mathrm{V}_{\overline R}((R:\overline R))$ holds. 

Now, $U\subset\overline R$ being minimal, there are at most two maximal ideals in $\overline R$ containing $(U:\overline R)$, so that $|\mathrm{MSupp}_{\overline R}(S/\overline R)|\leq 2$, with $|\mathrm{MSupp}_{\overline R}(S/\overline R)|$ $=2$ only if $U\subset\overline R$ is minimal decomposed (case (I)). This generalizes a result by Ben Nasr and Zeidi \cite[Corollary 2.10 and Remark 2.11]{BZ} gotten in the context of FCP extensions of integral domains. In case (I), let $N_1,N_2\in\mathrm{Max}(\overline R)$ be such that $(U:\overline R)=N_1\cap N_2\in\mathrm{Max}(U)$ and $\mathrm{MSupp}_{\overline R}(S/\overline R)=\{N_1,N_2\}$. This implies that $M:=(U:\overline R)\cap R=N_i\cap R\in\mathrm{MSupp}_R(\overline R/R)$. If other types of minimal extension $U\subset\overline R$ exist, we also have $M:=(U:\overline R)\cap R\in\mathrm{MSupp}_R(\overline R/R)$. If there does not exist some $U\in[R,\overline R]$ such that $U\subset\overline R$ is minimal decomposed, then, $|\mathrm{MSupp}_{\overline R}(S/\overline R)|=1$ (case (II)). Set $\{N\}:=\mathrm{MSupp}_{\overline R}(S/\overline R)$ in this case and $M:=N\cap R$.

Assume that  there exists some $U\in[R,S]$ such that $U\subset\overline R$ is minimal with $N'\in\mathrm{V}_{\overline R}((U:\overline R))$. Let $N''\in\mathrm{MSupp}_{\overline R}(S/\overline R)$, so that $N''=N_i$ for some $i=1,2$ in case (I) and $N''=N$ in case (II). There exists some $T\in[\overline R,S]$ such that $\overline R\subset T$ is minimal with crucial ideal $N''$ by \cite[Lemma 1.8]{Pic 6}. Then, $N'\cap U=(U:\overline R)\subseteq N''$. If $U\subset\overline R$ is minimal either inert or ramified, and then an $i$-extension, it follows that $N''=N'$ and $\mathrm{V}_{\overline R}((U:\overline R))=\{N''\}\subseteq\mathrm{MSupp}_{\overline R}(S/\overline R)$, so that $\mathrm{V}_{\overline R}((U:\overline R))=\mathrm{MSupp}_{\overline R}(S/\overline R)$. If $U\subset\overline R$ is minimal decomposed and $|\mathrm{MSupp}_{\overline R}(S/\overline R)|=2$, then $(U:\overline R)=N_1\cap N_2$ as we have just seen before; so that, $N'=N_i$ for some $i=1,2$. 
 Then, we still have $\mathrm{V}_{\overline R}((U:\overline R))=\mathrm{MSupp}_{\overline R}(S/\overline R)$. To conclude, $(U:\overline R)\cap R=M$ for any $U\in[R,\overline R]$ such that $U\subset\overline R$ is minimal.
 
We claim that, in any case, $|\mathrm{MSupp}_R(S/R)|=1$ with $\mathrm{MSupp}_R(S/R)=\{M\}$ where $M\in\mathrm{Max}(R)$ is such that either $M=N_i\cap R,\ i=1,2$ in case (I) or $M=N\cap R$ in case (II).  
 It is enough to show that $R\subset\overline R$ is an $M$-crucial extension because $\mathrm{MSupp}_R(S/\overline R)\subseteq \mathrm{V}_R((R:\overline R))=\mathrm{MSupp}_R(\overline R/R)$. 
 The last equality holds because $(R:\overline R)$ is the annihilator of the $R$-module $\overline R/R$ \cite[Proposition 17, page 133]{Bki A3} and $R\subset\overline R$ is integral. Assume that there exists $M'\in\mathrm{MSupp}_R(\overline R/R),\ M'\neq M$. Since $R\subseteq\overline R$ is an FCP integral extension, it is a $\mathcal B$-extension. Let $T'\in[R_{M'},\overline R_{M'}]$ be such that $T'\subset\overline R_{M'}$ is a minimal extension. In particular, $(T':\overline R_{M'})\cap R_{M'}=M'_{M'}$. There is a (unique) $T\in[R,\overline R]$ such that $T_{M'}=T'$ and $T_{M''}=\overline R_{M''}$ for any $M''\in\mathrm{Max}(R),\ M''\neq M'$. Then, $T\subset\overline R$ is minimal with $(T:\overline R)\cap R=M'$, a contradiction with $(V:\overline R)\cap R=M$ for any $V\in[R,\overline R]$ such that $V\subset\overline R$ is minimal as we have just proved above. Then, $\mathrm{MSupp}_R(\overline R/R)=\{M\}$, so that $|\mathrm{MSupp}_R(S/R)|=1$ and $R\subseteq\overline R$ is an $M$-crucial extension.

In order to find the other conditions, we make a discussion according to the cardinality of $\mathrm{MSupp}_{\overline R}(S/\overline R)$ and the canonical decomposition of $R\subset S$.

Case (I): $|\mathrm{MSupp}_{\overline R}(S/\overline R)|=2$. 
 This case happens if there exists some $U\in[R,\overline R]$ where $U\subset\overline R$ is minimal decomposed with $(U:\overline R)=N_1\cap N_2$ for some $N_1,N_2\in\mathrm{Max}(\overline R)$ such that $\mathrm{MSupp}_{\overline R}(S/\overline R)=\{N_1,N_2\}$ since $|\mathrm{MSupp}_{\overline R}(S/\overline R)|=2$. There cannot be some $V\in [R,S]$ such that $V\subset\overline R$ is either ramified or inert. Otherwise $(V:\overline R)$ is contained in only one maximal ideal of $\overline R$, a contradiction with $\mathrm{MSupp}_{\overline R}(S/\overline R)\subseteq\mathrm{V}_{\overline R}((V:\overline R))$. In particular, $R\subset\overline R$ is infra-integral since ${}_{\overline R}^tR=\overline R$. We claim that $U$ is the unique $V\in[R,\overline R]$ such that $(V:\overline R)=N_1\cap N_2$. If there exists $V\in[R,\overline R],\ V\neq U$ such that $(V:\overline R)=N_1\cap N_2$, then \cite[Proposition 5.7]{DPPS} shows that $N_1\cap N_2\in\mathrm{Max}(U\cap V)$. But $N_1\cap N_2\in\mathrm{Max}(U)$ with $(U\cap V)/(N_1\cap N_2)\cong U/(N_1\cap N_2)$ since $U\cap V\subset U$ is infra-integral, a contradiction.

 In particular, $\mathrm{V}_{\overline R}((R:\overline R))=\{N_1,N_2\}$ since $R\subseteq{}_{\overline R}^+R$ is an $i$-extension. It follows that $M\overline R$ is a radical ideal of $\overline R$ thanks to \cite[Lemma 17]{DPP4}. As $M\overline R\subseteq(U:\overline R)=N_1\cap N_2$  we get $M\overline R  =N_1\cap N_2$ and (1) is proved. 

Case (II): $|\mathrm{MSupp}_{\overline R}(S/\overline R)|=1$. Set 
 $\mathrm{MSupp}_{\overline R}(S/\overline R)=\{N\}$, so that $\overline R\subset S$ is $N$-crucial.  
 Then, any $U\in [R,S]$ such that $U\subset\overline R$ is minimal satisfies $N\in\mathrm{V}_{\overline R}((U:\overline R))$. If  
 $U\subset\overline R$ is minimal either inert or ramified, then  
 $\mathrm{V}_{\overline R}((U:\overline R))=\{N\}\ (*)$. 
  If $U\subset\overline R$ is minimal decomposed, we have necessarily  $(U:\overline R)=N\cap N'$ with $N'\in\mathrm{V}_{\overline R}((U:\overline R))$ and $(U:\overline R)=N\cap U=N'\cap U$ with $\mathrm{V}_{\overline R}((U:\overline R))=\{N,N'\}\ (**)$ and $N'\not\in\mathrm{MSupp}_{\overline R}(S/\overline R)$. In particular, $M=N\cap R$ since $\mathrm{MSupp}_R(S/\overline R)\subseteq\mathrm{MSupp}_R(S/ R)$.
    
Assume first that ${}_{\overline R}^tR\subset\overline R$, which is a t-closed extension and then an $i$-extension (Definition \ref{canonical}).

 We have $N\cap{}_{\overline R}^tR\in\mathrm{MSupp}_{{}_{\overline R}^tR}(\overline R/{}_{\overline R}^tR)$ because $(U:\overline R)=N$ for any $U\in[R,\overline R]$ such that $U\subset\overline R$ is minimal inert by $(*)$. Assume that there exists some $P\in\mathrm{MSupp}_{{}_{\overline R}^tR}(\overline R/{}_{\overline R}^tR)$ with $P\neq N\cap{}_{\overline R}^tR$. There is a unique $Q\in\mathrm{Max}(\overline R)$ lying over $P$, with $Q\neq N$, and successive applications of (CE) give some $U'\in[R,\overline R]$ such that $U'\subset\overline R$ is minimal inert with $(U':\overline R)=Q$, a contradiction by $(*)$. Then, $\mathrm{MSupp}_{{}_{\overline R}^tR}(\overline R/{}_{\overline R}^tR)=\{N\}$ since ${}_{\overline R}^tR\subset\overline R$ is t-closed, and then an $i$-extension, with $({}_{\overline R}^tR:\overline R)$ a radical ideal of $\overline R$. 

If there exists some $V\in[R,\overline R]$ such that $V\subset\overline R$ is minimal decomposed, we have $(V:\overline R)=N\cap N'$ with $N'\in\mathrm{V}_{\overline R}((V:\overline R))$ and $(V:\overline R)=N\cap V=N'\cap V$ by $(**)$. 

There is at most one $N'\in\mathrm{V}_{\overline R}((R:\overline R)),\ N'\neq N$ such that $N'\in\mathrm{V}_{\overline R}((V:\overline R))$ for some $V\in[R,S]$. Otherwise, if there exists another $N''\neq N$ corresponding to some $W\in[R,S]$ such that $W\subset\overline R$ is minimal decomposed with $(W:\overline R)=N\cap N''$, there also exists some $V'\in[R,S]$ such that $V'\subset\overline R$ is minimal decomposed with $(V':\overline R)=N'\cap N''$, because $\overline R/N'\cong\overline R/N\cong\overline R/N''$, a contradiction by Lemma \ref{1.132}. Then, there are no distinct $N',N''\neq N$ such that $\overline R/N'\not\cong\overline R/N''$.

In the same way, there cannot be some $U\in[R,S]$ such that $U\subset \overline R$ is minimal ramified with $(U:\overline R)$ contained in some $N''\neq N$ by $(*)$. According to Lemma \ref{7.9}, this is equivalent to $N''\not\in\mathrm{MSupp}_{\overline R}(N''/M\overline R)$, for any $N''\in\mathrm{V}_R((R:\overline R))$, with $N''\neq N$. In this case, (2) holds.

Assume now that ${}_{\overline R}^tR=\overline R$, so that $R\subset \overline R$ is infra-integral. Set $T:={}_{\overline R}^+R$. Assume also that $T\neq\overline R$. Set $C:=(T:\overline R)$ which is an intersection of finitely many maximal ideals of $\overline R$ by Proposition \ref{1.91}, and, more precisely, $C=\cap_{i=1}^nN_i$, for some integer $n\geq 2$, with, for instance $N_1:=N$  since $R\subseteq\overline R$ is an $M$-crucial extension.

Since $R\subseteq T$ is an $i$-extension and $R\subset\overline R$ is $M$-crucial, $C\in\mathrm{Max}(T)$ as the only maximal ideal of $T$ lying over $M$. Then, there is an order-isomorphism $[T,\overline R]\to[T_C,\overline R_C]$ given by $U\mapsto U_C$, which leads to $\overline R/C\cong\overline R_C/C_C\cong(T_C/C_C)^n\cong(T/C)^n$ by Proposition \ref{1.167} where $n:=\ell[T,\overline R]+1$. Then, $|\mathrm{V}_{\overline R}(C)|=n$. 

If $n\geq 3$, by Lemma \ref{1.132}, there exists $V\in[R,\overline R]$ such that $V\subset\overline R$ is minimal decomposed with $(V:\overline R)=N_i\cap N_j$ such that $i,j>1,\ i\neq j$, a contradiction by $(**)$. Then, $n=2$ and $T\subset\overline R$ is minimal decomposed. In particular, $(T:\overline R)=N\cap N'=N\cap T=N'\cap T$ is the only maximal ideal of $T$ lying over $M$ since $R\subseteq T$ is an $i$-extension. It follows that $\mathrm{V}_{\overline R}((R:\overline R))=\{N,N'\}$.

Assume that there exists some $W\in[R,\overline R]$ such that $W\subset\overline R$ is minimal ramified. Then, the only possible case is when $\mathrm{V}_{\overline R}((W:\overline R))=\{N\}$ by $(*)$. As $N'\cap W\neq N\cap W=(W:\overline R)$, it follows that there does not exist any $W\in[R,\overline R]$ such that $W\subset\overline R$ is minimal ramified with $(W:\overline R)\subset N'$, which is equivalent to 
 $N'\not\in\mathrm{MSupp}_{\overline R}(N'/M\overline R)$, by Lemma \ref{7.9}. In this case, (3) holds.

At last, if ${}_{\overline R}^+R=\overline R$, then $R\subset\overline R$ is subintegral, and (4) holds.

Conversely, assume that $|\mathrm{MSupp}_R(S/R)|=1$, 
  $\mathrm{MSupp}_{\overline R}(S/\overline R)\subseteq\mathrm{V}_{\overline R}((R:\overline R))$, and one of conditions (1), (2), (3) or (4) holds, where $\mathrm{MSupp}_R(S/R)=:\{M\}$.
  
Let $U\in[R,S]$ be such that $U\subset\overline R$ is minimal. Then, $(U:\overline R)\cap R=M$, because $(U:\overline R)\cap R\in\mathrm{MSupp}_R(S/R)$, and $(R:\overline R)\subseteq M$. We are going to show that $\mathrm{MSupp}_{\overline R}(S/\overline R)\subseteq\mathrm{V}_{\overline R}((U:\overline R))$ in each case,  according to 
Proposition \ref{1.131}.

(1) $|\mathrm{MSupp}_{\overline R}(S/\overline R)|=2,\ R\subseteq\overline R$ is infra-integral, with $\mathrm{MSupp}_{\overline R}(S/\overline R)$ $=\{N_1,N_2\}$ and $M\overline R=N_1\cap N_2$. 
 
It follows that $N_1$ and $N_2$ are the only maximal ideals of $\overline R$ lying over $M$. Since $R\subseteq\overline R$ is infra-integral, $U\subset\overline R$ cannot be inert. If $U\subset\overline R$ is minimal decomposed, it implies that $(U:\overline R)=N_1\cap N_2$, since $M\overline R\subseteq(U:\overline R)$ and $M\overline R=N_1\cap N_2$, which is a radical ideal of $\overline R$. Then, there cannot be any $U\in[R,S]$ such that $U\subset\overline R$ is minimal ramified according to \cite[Lemma 17]{DPP4}. It follows that, $\mathrm{MSupp}_{\overline R}(S/\overline R)\subseteq\mathrm{V}_{\overline R}((U:\overline R))$.

(2) $\mathrm{MSupp}_{\overline R}(S/\overline R)=\mathrm{MSupp}_{{}_{\overline R}^tR}(\overline R/{}_{\overline R}^tR)=\{N\},\ \overline R/N'\not\cong\overline R/N''\ (\dag)$ for any distinct $N',N''$ $\in\mathrm{V}_{\overline R}((R:\overline R))$, except for at most one pair $(N,N'),\ N'\neq N$ and, for any $N''\in\mathrm{V}_{\overline R}((R:\overline R))$ such that $ N''\neq N$, then $N''\not\in\mathrm{MSupp}_{\overline R}(N''/M\overline R)$.

Since $\mathrm{MSupp}_{{}_{\overline R}^tR}(\overline R/{}_{\overline R}^tR)=\{N\}$ and ${}_{\overline R}^tR\subset\overline R$ is a t-closed extension, $N=({}_{\overline R}^tR:\overline R)$ is both a maximal ideal of ${}_{\overline R}^tR$ and of $\overline R$ by Proposition \ref{1.9}.
 Moreover, $\mathrm{MSupp}_R(S/R)=\{M\}$ and $\mathrm{MSupp}_{{}_{\overline R}^tR}(\overline R/{}_{\overline R}^tR)=\{N\}$ implies that $M=N\cap R$. If $U\subset\overline R$ is minimal inert, then $(U:\overline R)=N$. If $U\subset\overline R$ is minimal decomposed, then $(U:\overline R)=N\cap N'$ by $(\dag)$, and there is only at most one $N'\neq N$ such that $\overline R/N\cong\overline R/N'$. Moreover, there are no distinct $N',N''\neq N$ such that $\overline R/N'\cong\overline R/N''$. Then, $(U:\overline R)\subseteq N$. At last, if $U\subset\overline R$ is minimal ramified, as $N''\not\in\mathrm{MSupp}_{\overline R}(N''/M\overline R)$, Lemma \ref{7.9} shows that $N$ is the only maximal ideal of $\overline R$ containing $(U:\overline R)$. In any case, we have $\mathrm{MSupp}_{\overline R}(S/\overline R)\subseteq\mathrm{V}_{\overline R}((U:\overline R))$.

(3) $|\mathrm{MSupp}_{\overline R}(S/\overline R)|=1,\ |\mathrm{V}_{\overline R}((R:\overline R))|=2$, with $\mathrm{MSupp}_{\overline R}(S/\overline R)\subset\mathrm{V}_{\overline R}((R:\overline R))$. Setting
$\mathrm{MSupp}_{\overline R}(S/\overline R)=\{N\}$ and $\mathrm{V}_{\overline R}((R:\overline R))=\{N,N'\},\ R\subseteq\overline R$ is infra-integral and $N'\not\in\mathrm{MSupp}_{\overline R}(N'/M{\overline R})$.

First, $U\subset\overline R$ cannot be minimal inert. Since ${}_{\overline R}^+R\subset\overline R$ is minimal decomposed, because $|\mathrm{V}_{\overline R}((R:\overline R))|=2$, ${}_{\overline R}^+R$ is the only $U\in[R,\overline R]$ such that $U\subset\overline R$ is minimal decomposed. We have necessarily $M=N\cap R$ because $\mathrm{MSupp}_R(S/\overline R)\subseteq\mathrm{MSupp}_R(S/ R)$.  
Then, $\mathrm{V}_{\overline R}((R:\overline R))=\{N,N'\}$ implies that $({}_{\overline R}^+R:\overline R)=N\cap N'$, so that $({}_{\overline R}^+R:\overline R)\subset N$. If $U\subset\overline R$ is minimal ramified, as $N'\not\in\mathrm{MSupp}_{\overline R}(N'/M{\overline R})$, Lemma \ref{7.9} shows that $N$ is the only maximal ideal of $\overline R$ containing $(U:\overline R)$. In any case, we have $\mathrm{MSupp}_{\overline R}(S/\overline R)\subseteq\mathrm{V}_{\overline R}((U:\overline R))$.

(4) $|\mathrm{MSupp}_{\overline R}(S/\overline R)|=1$ and $R\subseteq\overline R$ is subintegral.

First, $U\subset\overline R$ cannot be minimal either inert or decomposed. As in (3), we have necessarily $M=N\cap R$, where $\mathrm{MSupp}_R(S/\overline R)=:\{N\}$, because $\mathrm{MSupp}_R(S/\overline R)\subseteq\mathrm{MSupp}_R(S/ R)$. Since $R\subset\overline R$ is an $i$-extension, $N$ is the only maximal ideal of $\overline R$ lying above $M$. Then, $\mathrm{V}_{\overline R}((U:\overline R))=\{N\}$ and $\mathrm{MSupp}_{\overline R}(S/\overline R)\subseteq\mathrm{V}_{\overline R}((U:\overline R))$.

To sum up, in any case, $\mathrm{MSupp}_{\overline R}(S/\overline R)\subseteq\mathrm{V}_{\overline R}((U:\overline R))$, which shows that $R\subset S$ is pinched at $\overline R$.
\end{proof}

We deduce from Ayache's result \cite[Theorem 26]{A1} and Proposition \ref{7.10}, the following  Corollary:

\begin{corollary}\label{7.11}\cite[Theorem 26]{A1} Let $(R,M)$ be a local integral domain with quotient field $K$ such that $R\subset K$ has FCP. The following conditions are equivalent:
\begin{enumerate}
\item $R\subset K$ is pinched at $\overline R$;

\item  $R\subseteq\overline R$ is almost unbranched;

\item either $R\subset\overline R$ is infra-integral, ${}_{\overline R}^+R\subset\overline R$ is minimal decomposed and $R\subset\overline R$ is pinched at ${}_{\overline R}^+R$, or $R\subset \overline R$ is an i-extension.
\end{enumerate}
\end{corollary}

\begin{proof} (1) $\Leftrightarrow$ (2): By \cite[Theorem 26]{A1}, $R\subset K$ is pinched at $\overline R$ if and only if $T$ is local for any $T\in[R,\overline R[$, which is equivalent to $R\subseteq\overline R$ is almost unbranched. 

(1) $\Leftrightarrow$ (3): Since $(R:\overline R)\subseteq M$, we have $\mathrm{MSupp}_{\overline R}(K/\overline R)=\mathrm{V}_{\overline R}((R:\overline R))=\mathrm{Max}(\overline R)$. Now, using Proposition \ref{7.10}, $R\subset K$ is pinched at $\overline R$ if and only if $|\mathrm{MSupp}_R(K/R)|=1,\ \mathrm{MSupp}_{\overline R}(K/\overline R)\subseteq\mathrm{V}_{\overline R}((R:\overline R))$, conditions that always hold and one of  conditions (1)--(4) of Proposition \ref{7.10} holds: so that, (1) is equivalent to one of conditions (1)--(4) of Proposition \ref{7.10}. Then, it is enough to get conditions equivalent to conditions  (1)--(4) under the assumptions of the Corollary.

 Condition (1): $|\mathrm{MSupp}_{\overline R}(K/\overline R)|=2,\ R\subseteq\overline R$ is infra-integral and $M\overline R=N_1\cap N_2$, where $\mathrm{MSupp}_{\overline R}(K/\overline R)=\{N_1,N_2\}$. These conditions are equivalent to $R\subseteq\overline R$ is infra-integral, ${}_{\overline R}^+R\subset\overline R$ is minimal decomposed and $R\subset\overline R$ is pinched at ${}_{\overline R}^+R$, since there is no minimal ramified extension $U\subset \overline R$.
   
If one of Conditions (2)--(4) holds, $R\subset\overline R$ is always an $i$-extension since $|\mathrm{MSupp}_{\overline R}(K/\overline R)|$ $=1$ implies that $\overline R$ is a local ring. 

Conversely, assume that $R\subset\overline R$ is an $i$-extension. It follows that $\overline R$ is a local ring, so that $|\mathrm{MSupp}_{\overline R}(K/\overline R)|$ $=1$. Set $\mathrm{MSupp}_{\overline R}(K/\overline R)=\{N\}$. Then $\overline R\subset K$ is $N$-crucial.

If ${}_{\overline R}^tR\neq\overline R$, then 
$|\mathrm{MSupp}_{{}_{\overline R}^tR}(\overline R/{}_{\overline R}^tR)|=1$, with $\mathrm{MSupp}_{{}_{\overline R}^tR}(\overline R/{}_{\overline R}^tR)=\{N\}$, so that $|\mathrm{V}_{\overline R}((R:\overline R))|=1$. We get that Condition (2) of Proposition \ref{7.10} is obviously satisfied: $\overline R/N'\not\cong\overline R/N''$ since there are not distinct $N',N''\in\mathrm{V}_{\overline R}((R:\overline R))$, and, also any $N''\in\mathrm{V}_{\overline R}((R:\overline R))$ such that $ N',N''\neq N$.
 
If ${}_{\overline R}^tR=\overline R$, Condition (3) of Proposition \ref{7.10} cannot happen because $\overline R$ is a local ring, contradicting $|\mathrm{V}_{\overline R}((R:\overline R))|=2$. But Condition (4) of Proposition \ref{7.10} holds because
$|\mathrm{MSupp}_{\overline R}(K/\overline R)|=1$ and $R\subseteq\overline R$ is subintegral.

To sum up, $R\subset K$ is pinched at $\overline R$ if and only if either $R\subset\overline R$ is infra-integral, ${}_{\overline R}^+R\subset\overline R$ is minimal decomposed and $R\subset\overline R$ is pinched at ${}_{\overline R}^+R$ or $R\subset\overline R$ is an $i$-extension.
\end{proof}

 In \cite{DJ}, Dobbs and Jarboui introduced the notion of AV-ring pairs. If $R\subseteq S$ is a ring extension, $(R,S)$ is an {\em almost valuation ring pair} (or AV-ring pair) if, for any $T\in[R,S]$ and any $a,b\in T$, there exists an integer $n\geq 1$ such that either $a^nT\subseteq b^nT$ or $b^nT\subseteq a^nT$. They proved that when $(R,S)$ is an AV-ring pair such that $R\subseteq  \overline R$ has FCP, then $R\subset S$ is pinched at $\overline R$  \cite[Theorem 5.6]{DJ}. This leads us to the following example built from \cite[Example 3.3(a)]{DJ} by Dobbs-Jarboui.

\begin{example}\label{7.13} Let $F$ be a field with characteristic a prime integer $p$. Let $S:=F[[X]]$ be the ring of formal power series in the indeterminate $X$ and $K$ its quotient field. Set $R:=F[[X^p]]$. According to \cite[Example 3.3(a)]{DJ}, $R\subseteq K$ is an AV-ring pair and $S$ is a valuation domain, so that $S=\overline R$. Assume that $p>3$ and let $n$ be an integer such that $2<n<p$. Set $T:R+X^nS$. It follows that $T\subseteq K$ is an AV-ring pair such that $C:=(T:S)=X^nS$, and $S$ is also the integral closure $\overline T$ of $T$ in $K$. Since $T/C\cong R/(R\cap C)=F$, then $T/C$ is a field and $S$ is a finitely generated $T$-module with basis $\{1,\ldots,X^{n-1}\}$. Therefore, $T\subseteq S$ is an FCP integral extension. Then, \cite[Theorem 5.6]{DJ} shows that $T\subset K$ is pinched at $S$. We recover in this example case (4) of Proposition \ref{7.10} since $T\subset S$ is subintegral.
\end{example}

 \begin{remark}\label{AV pair}We can note that an AV-ring pair is quasi-Pr\"ufer. Let $R\subset S$ be an AV-ring pair. We may assume that $\overline R\neq R,S$ because if $\overline R=R,S$, the result is obvious. Then, $\overline R\subset S$ is integrally closed and also an AV-ring pair. It follows from \cite[ Proposition 2.2]{DJ} that $\overline R\subset S$ is a normal pair, or equivalently, a Pr\"ufer extension, so that $R\subset S$ is quasi-Pr\"ufer. Moreover, $R$ is unbranched in $S$ since $\overline R$ is local, thanks to \cite[Theorem 2.6]{DJ}. If, in addition, we assume that $R\subset S$ has FCP, with $\overline R\neq R,S$, we claim that $R\subset S$ is not almost-Pr\"ufer. Otherwise, $R\subset S$ splits at $\widetilde R\neq R$. Then, applying again  \cite[Theorem 2.6]{DJ} to the Pr\"ufer extension (then a normal pair) $R\subset\widetilde R$, we get that $R$ is a local ring, in contradiction with Theorem \ref{7.6} since $\mathrm{MSupp}(S/\widetilde{R})=\mathrm{MSupp}(\overline{R}/R)$ and $\mathrm{MSupp}(S/\widetilde{R})\cap\mathrm{MSupp}(\widetilde{R}/R)=\emptyset=\mathrm{MSupp}(\overline R/R)\cap\mathrm{MSupp}(\widetilde{R}/R)=\mathrm{Max}(R)$.
\end{remark}

In the context of  FIP extensions, the concepts of split extensions and  pinched extensions are mutually exclusive as it is explained in the next remark.

\begin{lemma}\label{7.14} Let $R\subset S$ be an FIP extension and $T\in]R,S[$. Then $R\subset S$ is pinched at $T$ if and only if  $|[R,S]|=|[R,T]|+|[T,S]|-1$.
\end{lemma}
\begin{proof} Since $[R,T]\cup[T,S]\subseteq [R,S]$ with $[R,T]\cap[T,S]=\{T\}$, we always have $|[R,S]|\geq|[R,T]|+|[T,S]|-1$. Then, $|[R,S]|>|[R,T]|+|[T,S]|-1$ if and only if there exists some $U\in[R,S]\setminus ([R,T]\cup[T,S])$, so that $|[R,S]|=|[R,T]|+|[T,S]|-1$ if and only if there does not exist any $U\in[R,S]\setminus([R,T]\cup[T,S])$ if and only if $R\subset S$ is pinched at  $T$.
\end{proof}

\begin{proposition}\label{7.15} Let $R\subset S$ be an FIP extension. Then, the following statements hold:
\begin{enumerate}
\item $|[R,\overline R]|+|[\overline R,S]|-1\leq|[R,S]|\leq|[R,\overline R]||[\overline R,S]|$;

\item $|[R,\overline R]|+|[\overline R,S]|-1=|[R,S]|$ if and only if $R\subset S$ is  pinched at $\overline R$;
\item $[R,S]|=|[R,\overline R]||[\overline R,S]|$ if and only if $R\subset S$   splits at $\overline R$.
\end{enumerate}
\end{proposition}
\begin{proof} (1) The first inequality comes from $[R,\overline R]\cup[\overline R,S]\subseteq[R,S]$, with $\overline R\in[R,\overline R]\cap[\overline R,S]$. The second inequality comes from Theorem \ref{4.10}.

(2) is Lemma \ref{7.14} applied to $T:=\overline R$ if $\overline R\neq R,S$. Otherwise, if $\overline R\in\{R,S\}$, the result is obvious (see Remark \ref{7.16} (2)).

(3) comes from Corollary \ref{4.6} and Proposition \ref{split}.
\end{proof}

\begin{remark}\label{7.16} (1) The three statements of Proposition \ref{7.15} show that when $R\subset S$ is an FIP extension, the least value of $|[R,S]|$ is gotten when $R\subset S$ is  pinched at $\overline R$, and its greatest value is gotten when $R\subset S$ is  split at $\overline R$. This means that the fact that $R\subset S$ is either pinched at $\overline R$ or  split at $\overline R$ leads to extreme situation.

(2) We have seen in Proposition \ref{1.145} that a splitter of $[R,S]$ is trivial when $R\subset S$ is pinched at some $U\in]R,S[$. In particular, if $R\subset S$ has FIP, the only case where $R\subset S$ is both pinched and split at $\overline R$ is when $|[R,\overline R]|+|[\overline R,S]|-1=|[R,\overline R]||[\overline R,S]|$. An easy calculation shows that this equation  is equivalent to either $\overline R=R$ or $\overline R=S$, which recovers Proposition \ref{1.145}.

(3) We have seen in Remark \ref{4.7} that there exists an FIP extension $R\subset S$ and $T\in]R,S[$ such that $|[R,S]|=|[R,T]||[T,S]|$ although $R\subset S$ is not split at $T$.

(4) As for the length, for an arbitrary FCP extension $R\subset S$, we always have $\ell[R,S]=\ell[R,\overline R]+\ell[\overline R,S]$ by \cite[Theorem 4.11]{DPP3}.
\end{remark}

\end{document}